\documentclass[11pt, oneside]{amsart}   	
\usepackage{geometry}                		
\usepackage[parfill]{parskip}    		
\usepackage{graphicx,tikz}				

\usepackage{amsmath, amsrefs, amssymb, amsthm}
\usepackage{hyperref}
\usepackage{mathtools}

\usepackage{caption}
\theoremstyle{plain}
\newtheorem{theorem}{Theorem}[section]
\newtheorem{proposition}[theorem]{Proposition}
\newtheorem{conjecture}[theorem]{Conjecture}
\newtheorem{lemma}[theorem]{Lemma}

\theoremstyle{definition}
\newtheorem{definition}[theorem]{Definition}

\newtheorem{example}[theorem]{Example}

\newcommand\partitionfr[1]{
	\coordinate (prev) at (0,0);
	\foreach \dir in {#1}{
		\draw[help lines, line width = .25mm] (prev) -- +(0,1) coordinate (prev);
		\draw[help lines, line width = .25mm] (prev)+(0,-1) grid +(\dir,0);
	};
}

\newcommand\dyckpathshade[3]{
	start point, size, Dyck word (size x 2 booleans)
	\fill[blue]  (#1) rectangle +(#2,#2);
	\fill[white]
	(#1)
	\foreach \dir in {#3}{
		\ifnum\dir=0
		-- ++(1,0)
		\else
		-- ++(0,1)
		\fi
	} |- (#1);
}
\newcommand\dyckpath[3]{

	\draw[help lines, gray!60, thin] (#1) grid +(#2,#2);
	\draw[dashed] (#1) -- +(#2,#2);
	\coordinate (prev) at (#1);
	\foreach \dir in {#3}{
		\ifnum\dir=0
		\coordinate (dep) at (1,0);
		\else
		\coordinate (dep) at (0,1);
		\fi
		\draw[line width=3pt,] (prev) -- ++(dep) coordinate (prev);
	};
}
\newcommand\dyckpathmn[4]{
	\draw[gray!60, thin] (#1) grid +(#2,#3);
	\coordinate (prev) at (#1);
	\foreach \dir in {#4}{
		\ifnum\dir=0
		\coordinate (dep) at (1,0);
		\else
		\coordinate (dep) at (0,1);
		\fi
		\draw[line width=3pt,black] (prev) -- ++(dep) coordinate (prev);
	};
}
\newcommand\dyckpathmnnogrid[4]{
	\coordinate (prev) at (#1);
	\foreach \dir in {#4}{
		\ifnum\dir=0
		\coordinate (dep) at (1,0);
		\else
		\coordinate (dep) at (0,1);
		\fi
		\draw[line width=3pt,black] (prev) -- ++(dep) coordinate (prev);
	};
}

\makeatletter
\renewcommand*{\eqref}[1]{%
  \hyperref[{#1}]{\textup{\tagform@{\ref*{#1}}}}%
}
\makeatother

\makeatletter

\pgfkeys{
	/tikz/sharp angle/.code={%
		\pgfsetarrowoptions{sharp >}{#1}%
		\pgfsetarrowoptions{sharp <}{-#1}%
	},
	/tikz/sharp > angle/.code={%
		\pgfsetarrowoptions{sharp >}{#1}%
	},
	/tikz/sharp < angle/.code={%
		\pgfsetarrowoptions{sharp <}{#1}%
	},
	/tikz/sharp protrude/.code=\csname if#1\endcsname\qrr@tikz@sharp@z@-0.05\p@\else\qrr@tikz@sharp@z@\z@\fi,
	/tikz/sharp protrude/.default=true
}

\newdimen\qrr@tikz@sharp@z@
\qrr@tikz@sharp@z@\z@
\pgfarrowsdeclare{sharp >}{sharp >}{%
	\edef\pgf@marshal{\noexpand\pgfutil@in@{and}{\pgfgetarrowoptions{sharp >}}}%
	\pgf@marshal
	\ifpgfutil@in@
	\edef\pgf@tempa{\pgfgetarrowoptions{sharp >}}
	\expandafter\qrr@tikz@sharp@parse\pgf@tempa\@qrr@tikz@sharp@parse
	\else
	\qrr@tikz@sharp@parse\pgfgetarrowoptions{sharp >}and-\pgfgetarrowoptions{sharp >}\@qrr@tikz@sharp@parse
	\fi
	\pgfmathparse{max(\pgf@tempa,\pgf@tempb,0)}%
	\let\qrr@tikz@sharp@max\pgfmathresult
	\pgfmathsetlength\pgf@xa{.5*\pgflinewidth * tan(\qrr@tikz@sharp@max)}%
	\pgfarrowsleftextend{+\pgf@xa}%
	\pgfarrowsrightextend{+\pgf@xa}%
}{%
	\edef\pgf@marshal{\noexpand\pgfutil@in@{and}{\pgfgetarrowoptions{sharp >}}}%
	\pgf@marshal
	\ifpgfutil@in@
	\edef\pgf@tempa{\pgfgetarrowoptions{sharp >}}
	\expandafter\qrr@tikz@sharp@parse\pgf@tempa\@qrr@tikz@sharp@parse
	\else
	\qrr@tikz@sharp@parse\pgfgetarrowoptions{sharp >}and-\pgfgetarrowoptions{sharp >}\@qrr@tikz@sharp@parse
	\fi
	\pgfmathsetlength\pgf@ya{.5*\pgflinewidth * tan(max(\pgf@tempa,\pgf@tempb,0))}%
	\pgfmathsetlength\pgf@xa{-.5*\pgflinewidth * tan(\pgf@tempa)}%
	\pgfmathsetlength\pgf@xb{-.5*\pgflinewidth * tan(\pgf@tempb)}%
	\advance\pgf@xa\pgf@ya
	\advance\pgf@xb\pgf@ya
	\ifdim\pgf@xa>\pgf@xb
	\pgftransformyscale{-1}%
	\pgf@xc\pgf@xb
	\pgf@xb\pgf@xa
	\pgf@xa\pgf@xc
	\fi
	\pgfpathmoveto{\pgfqpoint{\qrr@tikz@sharp@z@}{.5\pgflinewidth}}%
	\pgfpathlineto{\pgfqpoint{\pgf@xa}{.5\pgflinewidth}}%
	\pgfpathlineto{\pgfqpoint{\pgf@ya}{+0pt}}%
	\pgfpathlineto{\pgfqpoint{\pgf@xb}{-.5\pgflinewidth}}%
	\pgfpathlineto{\pgfqpoint{\qrr@tikz@sharp@z@}{-.5\pgflinewidth}}%
	\pgfusepathqfill
}
\pgfarrowsdeclare{sharp <}{sharp <}{%
	\edef\pgf@marshal{\noexpand\pgfutil@in@{and}{\pgfgetarrowoptions{sharp <}}}%
	\pgf@marshal
	\ifpgfutil@in@
	\edef\pgf@tempa{\pgfgetarrowoptions{sharp <}}
	\expandafter\qrr@tikz@sharp@parse\pgf@tempa\@qrr@tikz@sharp@parse
	\else
	\expandafter\qrr@tikz@sharp@parse\pgfgetarrowoptions{sharp <}and-\pgfgetarrowoptions{sharp <}\@qrr@tikz@sharp@parse
	\fi
	\pgfmathparse{max(\pgf@tempa,\pgf@tempb,0)}%
	\let\qrr@tikz@sharp@max\pgfmathresult
	\pgfmathsetlength\pgf@xa{.5*\pgflinewidth * tan(\qrr@tikz@sharp@max)}%
	\pgfarrowsleftextend{+\pgf@xa}%
	\pgfarrowsrightextend{+\pgf@xa}%
}{%
	\edef\pgf@marshal{\noexpand\pgfutil@in@{and}{\pgfgetarrowoptions{sharp <}}}%
	\pgf@marshal
	\ifpgfutil@in@
	\edef\pgf@tempa{\pgfgetarrowoptions{sharp <}}
	\expandafter\qrr@tikz@sharp@parse\pgf@tempa\@qrr@tikz@sharp@parse
	\else
	\expandafter\qrr@tikz@sharp@parse\pgfgetarrowoptions{sharp <}and-\pgfgetarrowoptions{sharp <}\@qrr@tikz@sharp@parse
	\fi
	\pgfmathsetlength\pgf@ya{.5*\pgflinewidth * tan(max(\pgf@tempa,\pgf@tempb,0))}%
	\pgfmathsetlength\pgf@xa{-.5*\pgflinewidth * tan(\pgf@tempa)}%
	\pgfmathsetlength\pgf@xb{-.5*\pgflinewidth * tan(\pgf@tempb)}%
	\advance\pgf@xa\pgf@ya
	\advance\pgf@xb\pgf@ya
	\ifdim\pgf@xa>\pgf@xb
	\pgftransformyscale{-1}%
	\pgf@xc\pgf@xb
	\pgf@xb\pgf@xa
	\pgf@xa\pgf@xc
	\fi
	\pgfpathmoveto{\pgfqpoint{\qrr@tikz@sharp@z@}{.5\pgflinewidth}}%
	\pgfpathlineto{\pgfqpoint{\pgf@xa}{.5\pgflinewidth}}%
	\pgfpathlineto{\pgfqpoint{\pgf@ya}{+0pt}}%
	\pgfpathlineto{\pgfqpoint{\pgf@xb}{-.5\pgflinewidth}}%
	\pgfpathlineto{\pgfqpoint{\qrr@tikz@sharp@z@}{-.5\pgflinewidth}}%
	\pgfusepathqfill
}
\def\qrr@tikz@sharp@parse#1and#2\@qrr@tikz@sharp@parse{\def\pgf@tempa{#1}\def\pgf@tempb{#2}}

\makeatother

\DeclareMathOperator{\area}{area}
\DeclareMathOperator{\dinv}{dinv}

\DeclareMathOperator{\sign}{sign}

\DeclareMathOperator{\maj}{maj}
\DeclareMathOperator{\comaj}{comaj}
\DeclareMathOperator{\revmaj}{revmaj}
\DeclareMathOperator{\revcomaj}{revcomaj}

\DeclareMathOperator{\Des}{Des}
\DeclareMathOperator{\asc}{asc}
\DeclareMathOperator{\Asc}{Asc}

\DeclareMathOperator{\weight}{weight}
\DeclareMathOperator{\Ht}{\widetilde{H}}

\DeclareMathOperator{\PR}{PR}
\DeclareMathOperator{\CR}{CR}
\DeclareMathOperator{\R}{R}
\DeclareMathOperator{\CC}{CC}

\DeclareMathOperator{\ith}{th}

\DeclareMathOperator{\bC}{ {\bf CC}}
\DeclareMathOperator{\LC}{ {LC}}

\DeclareMathOperator{\PF}{ {PF}}

\DeclareMathOperator{\WV}{ {WV}}
\DeclareMathOperator{\LW}{ {LW}}
\DeclareMathOperator{\spl}{ {split}}
\DeclareMathOperator{\join}{ {join}}
\DeclareMathOperator{\type}{ {type}}

\DeclareMathOperator{\SYT}{ {SYT}}
\DeclareMathOperator{\LPF}{ {LPF}}

\DeclareMathOperator{\PP}{\mathcal{P}}
\DeclareMathOperator{\LD}{LD}

\newcommand{\qbinom}[2]{\genfrac{[}{]}{0pt}{}{#1}{#2}}

\title{Delta and Theta Operator expansions}

\author{Alessandro Iraci}
\address{Universit\`a di Pisa \\ Dipartimento di Matematica}
\email{alessandro.iraci@unipi.it}

\author{Marino Romero}
\address{University of Pennsylvania \\ Department of Mathematics}
\email{mar007@sas.upenn.edu}

\begin{document}

\begin{abstract}
	We give an elementary symmetric function expansion for $M\Delta_{m_\gamma e_1}\Pi e_\lambda^{\ast}$ and $M\Delta_{m_\gamma e_1}\Pi s_\lambda^{\ast}$ when $t=1$ in terms of what we call $\gamma$-parking functions and lattice $\gamma$-parking functions.
	Here, $\Delta_F$ and $\Pi$ are certain eigenoperators of the modified Macdonald basis and $M=(1-q)(1-t)$.
	Our main results in turn give an elementary basis expansion at $t=1$ for symmetric functions of the form $M \Delta_{Fe_1} \Theta_{G} J$ whenever $F$ is expanded in terms of monomials, $G$ is expanded in terms of the elementary basis, and $J$ is expanded in terms of the modified elementary basis $\{\Pi e_\lambda^\ast\}_\lambda$.
	Even the most special cases of this general Delta and Theta operator expression are significant; we highlight a few of these special cases.
	We end by giving an $e$-positivity conjecture for when $t$ is not specialized, proposing that our objects can also give the elementary basis expansion in the unspecialized symmetric function.
\end{abstract}

\maketitle{}
\section{Introduction}

Delta and Theta operators, denoted by $\Delta_F$ and $\Theta_F$ for a choice of symmetric function $F$, are fundamental symmetric function operators in the theory of Macdonald polynomials.
Since their introduction, these operators have been shown to have incredible properties and connections to other areas of interest.
In introducing a brief history of these operators, we will point out some of these connections. For definitions of the symmetric functions discussed here, we refer the reader to Section \ref{Sectionsymmetricfunctions}.

Often, this area of study has three aspects. There is the symmetric function side, the representation theoretical side, and a combinatorial description. By giving Schur function expansions of the symmetric function side, one is able to give the multiplicities of irreducible representations in the representation theoretical side via the Frobenius map, which sends irreducible characters of the symmetric group to Schur functions:
Let $A^{\lambda}$ denote Young's irreducible representation of the symmetric group $S_n$ indexed by the partition $\lambda \vdash n$.
For any graded module
\begin{align*}
	V = \bigoplus_{\alpha} V_{\alpha} &  & \text{ with } &  & V_\alpha \simeq \bigoplus_{\lambda \vdash n} n^{\alpha}_\lambda A^\lambda,
\end{align*}
the graded Frobenius characteristic produces the symmetric function
\[
	\mathcal{F}(V) = \sum_{\lambda \vdash n} s_\lambda \sum_{\alpha} n_\lambda^\alpha Q^\alpha.
\]

On the other hand, when the combinatorial expansion of a symmetric function is Schur positive, it predicts the existence of a representation theoretical side. As we will describe here, when dealing with Macdonald polynomials, the representations associated to these expansions are often natural and important for a variety of areas of study.

As proved in \cite{Haiman-Vanishing-2002} and conjectured in \cite{Garsia-Haiman-qLagrange-1996}, $\Delta_{e_n} e_n$ gives the bigraded Frobenius characteristic for the space of $S_n$ coinvariants of the polynomial ring with two sets of commuting variables. More precisely, if $Y_n = y_1,\dots, y_n$ and $Z_n = z_1,\dots, z_n$ are two sets of commuting variables, then $\sigma \in S_n$ acts diagonally on the space of polynomials in $Y_n,Z_n$ by sending $y_i \mapsto y_{\sigma_i}$ and $z_i \mapsto z_{\sigma_i}$.
The space of diagonal coinvariants is given by the quotient
\[
	\mathcal{R}^{(2,0)} = \frac{\mathbb{C}[Y_n,Z_n]}{(\mathbb{C}[Y_n,Z_n]^{S_n}_+)},
\]
where $(\mathbb{C}[Y_n,Z_n]^{S_n}_+)$ is the ideal generated by $S_n$-invariants with no constant term. This space is $\mathbb{N}^2$ graded, and we can record the grading
by setting $Q^{(r,s)} = q^r t^s.$ Then Haiman's theorem states that
\[
	\Delta_{e_n} e_n = \mathcal{F} (\mathcal{R}^{(2,0)}).
\]

The symmetric function $\Delta_{e_n} e_n$ is most often denoted $\nabla e_n$, where $\nabla$ is the Bergeron-Garsia nabla operator defined in \cite{Bergeron-Garsia-ScienceFiction-1999}.
Haiman proves this equality through algebraic geometrical means, realizing this ring through the isospectral Hilbert scheme of points on the plane.  Hogancamp showed that the hook Schur functions in this symmetric function give the triply graded Khovanov-Rozansky
homology for $(n,n+1)$-torus knots. There is a more general statement involving $(n,nm \pm 1)$ torus knots, though we will not go into detail \cite{Hogancamp-2017}.

On the combinatorial side, there is the Shuffle theorem, conjectured in \cite{HHLRU-2005} and proved by Carlsson and Mellit \cite{Carlsson-Mellit-ShuffleConj-2018}. This conjecture stated that $\nabla e_n$ can be written as a sum over labeled Dyck paths.
Carlsson and Mellit in fact prove the compositional refinement conjectured in \cite{Haglund-Morse-Zabrocki-2012}.
Their methods introduced a Dyck Path Algebra. Mellit expanded this idea in order to prove the related Rational Shuffle theorem \cite{Mellit-Rational-2021}, and then showed that the  triply graded Khovanov-Rozansky homology for $(m,n)$-torus knots can be realized through the Elliptic Hall or Schiffman algebra \cite{Mellit-Torus-Knots-2017}.
On symmetric functions, this algebra can be generated by using the operators of multiplication by $e_1$ and $\Delta_{e_1}$. Theta operators can also be viewed as elements of this algebra.

The Delta conjecture \cite{Haglund-Remmel-Wilson-2018} gives a similar combinatorial description to the symmetric function $\Delta_{e_k} e_n$. Soon after, Zabrocki gave a corresponding $S_n$-module for this symmetric function, stating that
if we introduce a new set of anticommuting variables $T_n = \tau_1,\dots, \tau_n$, and set
\[
	\mathcal{R}^{(2,1)} =  \frac{\mathbb{C}[Y_n,Z_n, T_n]}{(\mathbb{C}[Y_n,Z_n,T_n]^{S_n}_+)},
\]
then
$\sum_{k} u^{n-k} \Delta_{e_{k-1}}'e_n$ gives $\mathcal{F}(\mathcal{R}^{(2,1)} )$, the triply graded Frobenius characteristic for the space of $S_n$ coinvariants in two sets of commuting variables and one set of anti-commuting variables.

The methods used by Carlsson and Mellit in the proof of the shuffle theorem relied on the compositional refinement of the statement; Theta operators were then introduced in \cite{DAdderio-Iraci-VandenWyngaerd-Theta-2021} in order to give a compositional refinement of the Delta conjecture, which ultimately led to a proof of the compositional Delta theorem \cite{DAdderio-Mellit-Compositional-Delta-2020}.
Most recently, the extended Delta conjecture was also proved in \cite{Blasiak-Haiman-Morse-Pun-Seeling-Extended-Delta-2021}, giving the combinatorial description for $\Delta_{h_a} \Delta_{e_{k-1}}' e_n$. This is realized through a connection to $GL_m$ characters and the $LLT$ polynomials of \cite{Lascoux-Leclerc-Thibon-1997}.

If we introduce yet another set of anticommuting variables and let $\mathcal{R}^{(2,2)}$ be the $S_n$ coinvariants with two sets of commuting and two set of anticommuting variables, then
it was also conjectured in \cite{DAdderio-Iraci-VandenWyngaerd-Theta-2021} that
\begin{equation}
	\label{eq:22character}
	\mathcal{F}(\mathcal{R}^{(2,2)}) = \sum_{r,s \geq 0} u^r v^s \Theta_{e_r} \Theta_{e_ s} \nabla e_{n-r-s},
\end{equation}
meaning the Frobenius characteristic of $\mathcal{R}^{(2,2)}$ is given via Theta operators.
The purely fermionic case $\mathcal{R}^{(0,2)}$, involving only the portion with anticommuting variables (obtained by setting $q=t=0$ in \eqref{eq:22character}) has recently been proved in \cite{Iraci-Rhoades-Romero-2022}. For the $\mathcal{R}^{(1,1)}$ case (found by setting $t=u=0$ in \eqref{eq:22character}) the graded dimension of the coinvariant space with one set of commuting variables and one set of anticommuting variables has been shown in \cite{Rhoades-Wilson-2023} to agree with the conjectured formula.
\\

In general, the Schur function expansion of $\Delta_{e_k} e_n$ is unknown, and similarly, $\Theta_{e_\lambda} = \Theta_{e_{\lambda_1}} \cdots \Theta_{e_{\lambda_{\ell(\lambda)}}}  $ is yet to be fully understood at the combinatorial level. For the first main result of this paper, we give a combinatorial expansion for symmetric functions of the form \[ M \Delta_{F e_1} \Theta_{e_\lambda} G \qquad \qquad (\text{where}~M = (1-q)(1-t)), \] when $t=1$, in terms of the elementary symmetric function basis (precise definitions are given in Section \ref{Sectionsymmetricfunctions}). An elementary basis expansion gives a Schur expansion by simply using the Pieri rule.
If the original symmetric function is positive in some basis, then setting $t=1$ (or $q=1$) leaves the ungraded multiplicities intact. Therefore, giving an expansion at $t=1$ would predict the combinatorial objects enumerated by these symmetric functions without the specialization.
Even more, we find that certain symmetric functions are not Schur positive, yet become positive in the elementary basis when $t=1$. And even more surprising, we have Conjecture \ref{conjecture-epositivity} which predicts that this symmetric function is $e$-positive after substituting $q=1+u$ (rather than substituting $t=1$) for suitable $F$ and $G$.

The main strategy of our work is to expand the symmetric function, when $t=1$, as a series in $q$. One of the amazing aspects of this method, found in \cite{Hicks-Romero-2018}, is the use of the combinatorial formula for forgotten symmetric functions and their principal evaluation. The terms in the series are sums of certain signed combinatorial objects. After applying a weight-preserving, sign-reversing involution, we are able to get a finite number of positive fixed points, which bijectively correspond to some set of labeled polyominoes. The end result is found by adjusting the polyomino picture to get an expansion in terms of what we call $\gamma$-parking functions:
\begin{theorem}
	\label{Theorem1}
	For any two partitions $\lambda$ and $\gamma$, there is a family of labeled polyominoes $\PF^{\gamma}_{\lambda}$, called $\gamma$-parking functions of content $\lambda$, and a statistic $\area$ giving
	\begin{equation}
		\left. \Delta_{m_\gamma} M \Delta_{e_1} \Pi e_{\lambda}^\ast \right\rvert_{t=1} = \sum_{p \in \PF^{\gamma}_{\lambda} } q^{\area(p)} e_{\eta(p)}.
	\end{equation}
\end{theorem}
This solves the problem of computing $\Delta_{F e_1} M \Theta_{e_\lambda} G $ whenever $F$ is expanded in terms of the monomial basis and $G$ is expanded in terms of the modified $e$-basis $\{\Pi e_\lambda^{\ast}\}_\lambda$.
Using the same methods that prove Theorem \ref{Theorem1}, we also show

\begin{theorem}
	\label{Theorem2}
	For any two partitions $\lambda$ and $\gamma$, there is a family of labeled polyominoes $\LPF^{\gamma}_{\lambda}$, called lattice $\gamma$-parking functions of content $\lambda$, and a statistic $\area$ giving
	\begin{equation}
		\left. \Delta_{m_\gamma} M \Delta_{e_1}   \Pi s_{\lambda}^\ast \right\rvert_{t=1} =  \sum_{p \in \LPF^{\gamma}_{\lambda'} } q^{\area(p)} e_{\eta(p)}.
	\end{equation}
\end{theorem}
This raises the interesting problem of finding a monomial expansion for the same expression, in the same fashion as the one given in \cite{DAdderio-Iraci-LeBorgne-Romero-VandenWyngaerd-2022} in terms of tiered trees.

\section{Combinatorial definitions}
\label{gammaparkingfunctions}

In this section we aim to introduce the combinatorial objects that will give us the symmetric function expansions we are interested in.

\subsection{Words}

\begin{definition}
	A word of length $r$ is an element $w = (w_1, \dots, w_r) \in \mathbb{N}^r$. We denote the length by $\ell(w) = r$ and the size by $\lvert w \rvert = \sum_i w_i$.
\end{definition}

Let $w$ be a word of length $r$. We define the \emph{descent set} as $\Des(w) \coloneqq \{ 1 \leq i < r \mid w_i > w_{i+1} \}$, and the \emph{ascent set} as $\Asc(w) \coloneqq \{ 1 \leq i < r \mid w_i < w_{i+1} \}$.

We have the following statistics.

\begin{align*}
	 & \maj(w)    &  & = \sum_{i \in \Des(w)} i, \qquad \qquad     &  & \comaj(w)    &  & = \sum_{i \in \Des(w)} (n-i), \\
	 & \revmaj(w) &  & = \sum_{i \in \Asc(w)} (n-i), \qquad \qquad &  & \revcomaj(w) &  & = \sum_{i \in \Asc(w)} i.
\end{align*}

Note that $\revmaj$ and $\revcomaj$ actually are the $\maj$ and $\comaj$ of the reverse word, hence the name.

Let $m_i(w)$ be the number of indices $j$ such that $w_j = i$, that is, $m_i(w)$ is the multiplicity of $i$ in $w$. We denote the multiplicity type of $w$ as $m(w) = 0^{m_0(w)}1^{m_1(w)} 2^{m_2(w)} \cdots$. If $w \in \mathbb{N}^r_+$ (it has no $0$ entries), then we call it a composition and write $w \vDash \lvert w \rvert$.

There is a class of words that is of special interest to us.

\begin{definition}
	A \emph{lattice word} is a word $w = (w_1, \dots, w_r) \in \mathbb{N}_+^r$ such that, for all $1 \leq i, j \leq r$, we have \[
		m_{j+1}(w_1,\dots,w_i) \leq m_j(w_{1},\dots, w_{i}), \] that is, a word such that every prefix has at least as many $1$s as $2$s, at least as many $2$s as $3$s, and so on.
\end{definition}

Denote by $R(w)$ the set of all words $\alpha = (\alpha_1, \dots, \alpha_r)$ whose entries can be rearranged to give $w$, or $m(\alpha) = m(w)$.  If $\alpha_1 \geq \alpha_2 \geq \cdots \geq \alpha_r > 0$, then $\alpha$ is a partition, written $\alpha \vdash \lvert \alpha \rvert$.

It will be convenient to write a sequence of words $\vec{w}= (w^1,\dots, w^r)$, with $w^i \in \mathbb{N}^{r_i}$, as a vector. The type of $\vec{w}$, denoted  by $m(\vec{w})$, is the multiplicity type of the concatenation $w^1 \cdots  w^r = (w^1_1, w^1_2, \dots, w^2_1, w^2_2, \dots, \dots).$

We define the sets of \emph{word vectors of length $\beta$ and content $\alpha$}, \emph{composition vectors of size $\beta$ rearranging to $\alpha$}, and \emph{partition vectors of size $\beta$ rearranging to $\alpha$} as
\begin{align*}
	\WV(\alpha,\beta) & = \{ \vec{w} \mid \ell(w^i) = \beta_i ~ \text{ and } ~m(\vec{w}) = 0^{|\beta|-\ell(\alpha)} 1^{\alpha_1} 2^{\alpha_2} \cdots\} \\
	\CR(\alpha,\beta) & = \{ \vec{w} \mid w^i \vDash \beta_i ~ \text{ and } ~ w^1\cdots w^{\ell(w) } \in R(\alpha) \}                                  \\
	\PR(\alpha,\beta) & = \{ \vec{w} \mid w^i \vdash \beta_i ~ \text{ and } ~ w^1\cdots w^{\ell(w)} \in R(\alpha) \}
\end{align*}

The first is the set of sequences of words where the collective multiplicity of $i$ is $\alpha_i$ and sequence $j$ has length $\beta_j$. If $\ell(\alpha) < \lvert \beta \rvert$, then it is impossible to do this without allowing $0$ entries, of which there must be $|\beta|-\ell(\alpha)$.
The second set is the sequence of compositions whose sizes are determined by $\beta$ and whose parts collectively rearrange to $\alpha$; and the last set is the set of sequences of partitions whose sizes are determined by $\beta$ and whose collective union of parts rearranges to $\alpha$.

We represent partitions by their Young diagram. For a partition $\mu$ and a cell $c \in \mu$, we let $a(c),l(c), a'(c), $ and $l'(c)$ denote the arm, leg, coarm, and coleg of the cell. This gives the number of cells in $\mu$ strictly to the right, above, to the left, and below of $c$, respectively. See Figure~\ref{fig:(1)-armslegs} for an example.

\begin{figure}[!ht]
	\centering
	\begin{tikzpicture}[scale=1]
		\partitionfr{4,4,3}
		\draw (2.5,1.5) node {\Large$c$};
		\draw (3.5,1.5) node {\Large$a$};
		\draw (.5,1.5) node {\Large$a'$};
		\draw (1.5,1.5) node {\Large$a'$};
		\draw (2.5,.5) node {\Large$l'$};
		\draw (2.5,2.5) node {\Large$l$};
	\end{tikzpicture}
	\caption{The partition $(3,3,2)$, where we highlight the cells in the arm, leg, coarm, and coleg of the cell $c = (3,2)$ with $a,l,a',$ and $l'$, respectively}
	\label{fig:(1)-armslegs}
\end{figure}
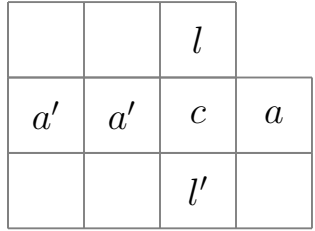

\subsection{\texorpdfstring{$\gamma$}{gamma}-Dyck paths}
We need to recall this classical definition.

\begin{definition}
	\label{def:polyomino}
	A \emph{parallelogram polyomino} of size $m \times n$ is a pair of lattice paths $(P,Q)$ from $(0,0)$ to $(m,n)$, consisting of unit North and East steps, such that $P$ (the \emph{top path}) lies always strictly above $Q$ (the \emph{bottom path}), except on the endpoints.
\end{definition}

The \emph{area} of a parallelogram polyomino of size $m \times n$ is defined as \[ \area(P,Q) = (\#\text{ of lattice cells between $P$ and $Q$}) - (m+n-1). \] Since the two paths $P$ and $Q$ do not touch between the endpoints, $m+n-1$ is the minimal number of unit cells between them.

\begin{figure}[!ht]
	\centering
	\begin{tikzpicture}[scale = 0.7]
		\draw[draw=none, use as bounding box] (-1, -1) rectangle (16,9);
		\draw[gray!60, thin] (0,0) grid (15,8);

		\filldraw[yellow, opacity=0.3] (0,0) -- (1,0) -- (2,0) -- (3,0) -- (4,0) -- (4,1) -- (4,2) -- (5,2) -- (5,3) -- (6,3) -- (7,3) -- (8,3) -- (9,3) -- (9,4) -- (10,4) -- (11,4) -- (11,5) -- (11,6) -- (12,6) -- (13,6) -- (14,6) -- (14,7) -- (15,7) -- (15,8) -- (14,8) -- (13,8) -- (13,7) -- (12,7) -- (11,7) -- (10,7) -- (9,7) -- (8,7) -- (7,7) -- (7,6) -- (7,5) -- (6,5) -- (5,5) -- (4,5) -- (3,5) -- (2,5) -- (2,4) -- (2,3) -- (1,3) -- (0,3) -- (0,2) -- (0,1) -- (0,0);

		\draw[red, sharp <-sharp >, sharp angle = -45, line width=1.6pt] (0,0) -- (0,1) -- (0,2) -- (0,3) -- (1,3) -- (2,3) -- (2,4) -- (2,5) -- (3,5) -- (4,5) -- (5,5) -- (6,5) -- (7,5) -- (7,6) -- (7,7) -- (8,7) -- (9,7) -- (10,7) -- (11,7) -- (12,7) -- (13,7) -- (13,8) -- (14,8) -- (15,8);

		\draw[green, sharp <-sharp >, sharp angle = 45, line width=1.6pt] (0,0) -- (1,0) -- (2,0) -- (3,0) -- (4,0) -- (4,1) -- (4,2) -- (5,2) -- (5,3) -- (6,3) -- (7,3) -- (8,3) -- (9,3) -- (9,4) -- (10,4) -- (11,4) -- (11,5) -- (11,6) -- (12,6) -- (13,6) -- (14,6) -- (14,7) -- (15,7) -- (15,8);
	\end{tikzpicture}
	\caption{A parallelogram polyomino with area $20$}
\end{figure}

We can now introduce our new objects.

\begin{definition}
	Let $\gamma \vdash m$. A \emph{$\gamma$-Dyck path} of size $n$ is a parallelogram polyomino of size $(m + n + 1) \times n$ such that the bottom path does not have two consecutive North steps, and if $\alpha_i$ is the number of East step of the bottom path in the $i^{\ith}$ row, then $(\alpha_1 - 1, \alpha_2, \dots, \alpha_n)$ rearranges to $\gamma + 1^n = (\gamma_1+1,\dots,\gamma_{\ell(\gamma)}+1,1,\dots,1)$.
\end{definition}

Notice that $\varnothing$-Dyck paths are essentially the same thing as classical Dyck paths, as the condition imposed by the bottom path is that the top path lies always weakly above the diagonal $x=y$.

\begin{definition}
	A \emph{labeled $\gamma$-Dyck path}, is a $\gamma$-Dyck path in which each North step is assigned a positive integer label such that consecutive North steps are assigned strictly increasing labels.
	A labeled $\gamma$-Dyck path will be denoted as a triple $p = (P,Q,w)$, where $P$ is the top path, $Q$ is the bottom path, and $w$ is the word formed by the labels when read from bottom to top.
	The \emph{content} of a labeled $\gamma$-Dyck path is the weak composition $\alpha \vDash_w n$ whose parts $\alpha_i$ give the number of $i$'s appearing in the labeling (or $m(w) =0^{n-\ell(\alpha)}1^{\alpha_1}2^{\alpha_2}\dots$).
	A $\gamma$-parking function is a labeled $\gamma$-Dyck path of content $1^n$. For our convenience, we will also refer to labeled $\gamma$-Dyck paths of content $\alpha$ as $\gamma$-parking functions of content $\alpha$, and denote them by $\PF_\alpha^\gamma$.
\end{definition}

\begin{figure}[!ht]
	\centering
	\includegraphics{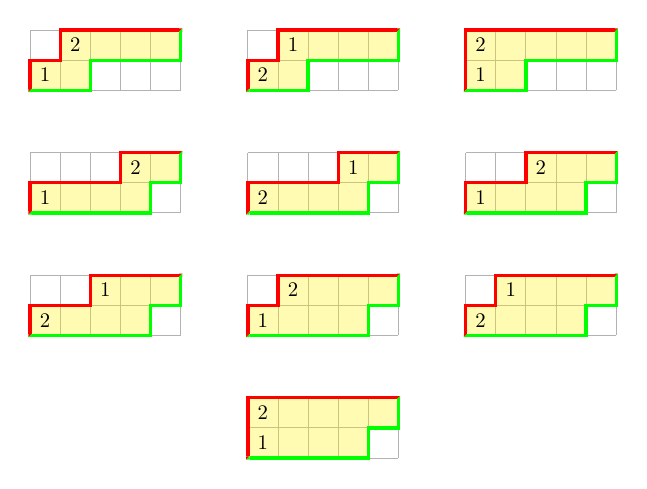}
	\caption{The set of $(2)$-parking functions of height $2$}
	\label{fig:(2)-parking-functions}
\end{figure}

\begin{definition}
	A \emph{lattice $\gamma$-Dyck path} is a labeled $\gamma$-Dyck path in which the sequence of labels, read bottom to top, is a lattice word. Notice that the content of a lattice word is necessarily a partition. As above, for our convenience, we will also refer to lattice $\gamma$-Dyck paths with content $\lambda$ as lattice $\gamma$-parking functions with content $\lambda$, and denote them by $\LPF_\lambda^\gamma$.
\end{definition}

\begin{definition}
	\label{def:e-composition}
	The \emph{$e$-composition} $\eta(p)$ of a labeled $\gamma$-Dyck path $p = (P,Q,w)$ is defined as follows: let $\overline{P}$ be the path obtained from $P$ by removing the first East step after the $i^{\ith}$ North step for every $i \not \in \Asc(w)$; $\eta(p)$ is the composition whose parts are the lengths of the maximal sequences of consecutive North steps appearing in $\overline{P}$, from the bottom to the top.
	An example is given in Figure~\ref{fig:e-comp}.
\end{definition}

\begin{figure}[!ht]
	\centering
	\begin{align*}
		\scalebox{.5}{
			\begin{tikzpicture}[scale=1]
				\draw[gray!60, thin] (0,0) grid (13,7);
				\filldraw[yellow, opacity=0.3] (0,0) -- (1,0) -- (2,0) -- (3,0) -- (3,1) -- (4,1) -- (4,2) -- (5,2) -- (6,2) -- (7,2) -- (7,3) -- (8,3) -- (9,3) -- (9,4) -- (10,4) -- (10,5) -- (11,5) -- (12,5) -- (12,6) -- (13,6) -- (13,7) -- (6,7) -- (5,7) -- (5,6) -- (4,6) -- (3,6) -- (3,5) -- (3,4) -- (2,4) -- (1,4) -- (1,3) -- (0,3) -- (0,2) -- (0,1) -- (0,0);
				\draw[red, sharp <-sharp >, sharp angle = -45, line width=1.6pt] (0,0) -- (0,1) -- (0,2) -- (0,3) -- (1,3) -- (1,4) -- (2,4) -- (3,4) -- (3,5) -- (3,6) -- (4,6) -- (5,6) -- (5,7) -- (6,7) -- (13,7);
				\draw[green, sharp <-sharp >, sharp angle = 45, line width=1.6pt] (0,0) -- (1,0) -- (2,0) -- (3,0) -- (3,1) -- (4,1) -- (4,2) -- (5,2) -- (6,2) -- (7,2) -- (7,3) -- (8,3) -- (9,3) -- (9,4) -- (10,4) -- (10,5) -- (11,5) -- (12,5) -- (12,6) -- (13,6) -- (13,7);
				\node at (0.5,0.5) {\LARGE$1$};
				\node at (0.5,1.5) {\LARGE$4$};
				\node at (0.5,2.5) {\LARGE$5$};
				\node at (1.5,3.5) {\LARGE$3$};
				\node at (3.5,4.5) {\LARGE$1$};
				\node at (3.5,5.5) {\LARGE$2$};
				\node at (5.5,6.5) {\LARGE$5$};
				\node at (0.5, 3) {\Huge $\times$};
				\node at (1.5, 4) {\Huge $\times$};
				\node at (5.5, 7) {\Huge $\times$};
			\end{tikzpicture}
		}
		 &  &
		\scalebox{.5}{
			\begin{tikzpicture}[scale=1]
				\draw (0,0) node { };
				\draw (0,3.5) node {\LARGE$\longrightarrow$};
			\end{tikzpicture}
		}
		 &  &
		\scalebox{.5}{
			\begin{tikzpicture}[scale=1]
				\draw[gray!60, thin] (0,0) grid (6,7);
				\draw[red, sharp <-sharp >, line width=1.6pt] (0,0) -- (0,1) -- (0,2) -- (0,3) -- (0,4) -- (1,4) -- (1,5) -- (1,6) -- (2,6) -- (3,6) -- (3,7) -- (4,7);
				\draw[red, dashed, line width=1.6pt] (4,7) -- (6,7);
				\node at (0.5,0.5) {\LARGE$1$};
				\node at (0.5,1.5) {\LARGE$4$};
				\node at (0.5,2.5) {\LARGE$5$};
				\node at (0.5,3.5) {\LARGE$3$};
				\node at (1.5,4.5) {\LARGE$1$};
				\node at (1.5,5.5) {\LARGE$2$};
				\node at (3.5,6.5) {\LARGE$5$};
			\end{tikzpicture}
		}
	\end{align*}
	\caption{The algorithm to determine the $e$-composition of a labeled $\gamma$-Dyck path. In this case, $\eta(p) = (4,2,1)$}
	\label{fig:e-comp}
\end{figure}
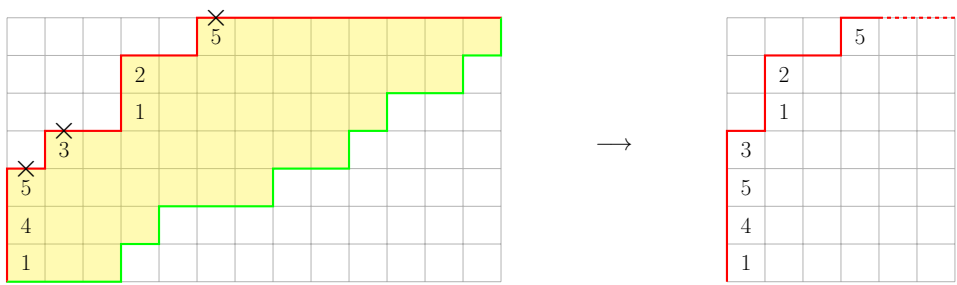

\section{Symmetric function preliminaries}
\label{Sectionsymmetricfunctions}

The standard reference for Macdonald polynomials is Macdonald's book \cite{Macdonald-Book-1995}. For some reference on modified Macdonald polynomials, plethystic substitution, and Delta operators, we have \cite{Haglund-Book-2008} and \cite{Bergeron-Garsia-Haiman-Tesler-Positivity-1999}. As a reference for Theta operators, we have \cite{DAdderio-Iraci-VandenWyngaerd-Theta-2021} and \cite{DAdderio-Romero-Theta-Identities-2020}.

From here on, we set $M=(1-q)(1-t)$, and for any $\mu$ define
\begin{align*}
	\Pi_\mu  = \prod_{c \in \mu/(1)} & 1-q^{a'(c)} t^{l'(c)}, \hspace{2.5cm}
	B_\mu = \sum_{c \in \mu} q^{a'(\mu)} t^{l'(\mu)},                                                                                   \\
	                                 & w_\mu = \prod_{c \in \mu} \left(q^{a(c)} -t^{l(c)+1} \right) \left(t^{l(c)} -q^{a(c)+1} \right).
\end{align*}

Recall the ordinary Hall scalar product gives the orthogonality relation \[ \langle s_\lambda, s_\mu \rangle =  \chi(\lambda = \mu), \] where $\chi(A) = 1$ if $A$ is true, and $0$ otherwise. The $\ast$-scalar product may be given by setting for any two symmetric functions $F$ and $G$, \[ \langle F, G \rangle_\ast= \langle F, (\omega G) [M X] \rangle \] where $\omega$ is the algebra isomorphism on symmetric functions defined by $\omega(e_n) = h_n$. Note that $\omega$ is also an isometry and a Hopf algebra antipode (hence an involution).

We can now state the orthogonality relations on the modified Macdonald basis: \[ \langle \Ht_\lambda  , \Ht_\mu \rangle_\ast = w_\mu \chi(\lambda = \mu). \]

The Delta operators are eigenoperators of the modified Macdonald basis indexed by symmetric functions, defined by setting $\Delta_{F} \Ht_\mu  = F[B_\mu] \Ht_\mu$ \cite{Bergeron-Garsia-Haiman-Tesler-Positivity-1999}. For any symmetric functions $F$ and $G$, let \[ \widetilde{\Delta}_F G  = \left. \Delta_F G \right\rvert_{t=1}. \]

On the space of symmetric functions with coefficients in $q$, the operator $\widetilde{\Delta}_F$ can be defined by setting \[ \widetilde{\Delta}_F h_\mu\left[ \frac{X}{1-q} \right] = F\left[ \sum_{i=1}^{\ell(\mu)} [\mu_i]_q \right] h_\mu \left[ \frac{X}{1-q} \right]. \]

This statement follows from the fact that the modified Macdonald basis specializes as follows: \[ \Ht_\mu[X;q,1] = (q;q)_\mu h_\mu \left[ \frac{X}{1-q} \right], \] where \[ (q;t)_r = (1-q)(1-qt) \cdots (1-qt^{r-1}) \] is the $q$-Pochhammer symbol, and $(q;q)_\mu = (q;q)_{\mu_1} \cdots (q;q)_{\mu_{\ell(\mu)}}$.

To define Theta operators, we must define another eigenoperator of the modified Macdonald basis. Let $\Pi$ be the linear operator defined by setting $\Pi \Ht_\mu = \Pi_\mu \Ht_\mu.$ For any symmetric function $F$, let \[ F^\ast[X] = F\left[ \frac{X}{M} \right]. \] Then we can define Theta operators by setting \[ \Theta_F G =  \Pi F^\ast \Pi^{-1} G. \]

It has become apparent recently, that in order for Theta operator identities to hold in generality, we need the convention that $\Theta_F G = 0$ whenever the degree of $G$ is $0$ and the degree of $F$ is larger than $0$. However, we would like to study the symmetric function $\Delta_{e_1} M\Theta_{e_\lambda} \Pi e_\mu^\ast$, which, without this convention, would look like $M \Delta_{e_1} \Theta_{e_\lambda e_\mu}(1)$, where the $\Delta_{e_1}$ is needed to get positive expressions.
In the end, it is convenient to define the symmetric function operator $\Xi$ by setting \[ \Xi F = M \Delta_{e_1} \Pi F^{\ast}[X]. \]

Our first goal will be to give an elementary basis expansion of  $\widetilde{\Delta}_{m_\gamma} \Xi e_{\lambda}$.

\section{Preliminary manipulations and specializations}
\label{sec:preliminary-manipulations}

The first step in studying $\Xi e_\lambda$, for $\lambda \vdash n$, is to note that if we want to expand in terms of the modified Macdonald basis, we first have
\[ e_{\lambda}^\ast = \sum_{\mu \vdash n} \frac{\Ht_\mu}{w_\mu} \langle e_{\lambda}^\ast , \Ht_\mu \rangle_\ast = \sum_{\mu \vdash n} \frac{\Ht_\mu}{w_\mu} \langle h_{\lambda} , \Ht_\mu \rangle, \]
where in the last step we went from the $\ast$-scalar product to the ordinary Hall scalar product. Applying the list of operators, we then have that
\[ \Delta_{m_\gamma} M \Delta_{e_1} \Pi e_{\lambda}^\ast   = \sum_{\mu \vdash n} \frac{ M B_\mu \Pi_\mu}{w_\mu}   \langle  h_{\lambda}  , \Ht_\mu \rangle m_\gamma[B_\mu] \Ht_\mu. \]

We will break up these summation terms by analyzing each of the three components, specializing $t$ to $1$ for each one individually. For the first component, we first separate the product into two parts:
\begin{align*}
	\frac{MB_\mu \Pi_\mu}{w_\mu} & = \frac{ (1-q) B_\mu  \prod_{ \substack{ c\in \mu    \\ a'(c) \neq 0}} (1-q^{a'(c)}t^{l'(c)})}
	{\prod_{\substack {c\in \mu                                                         \\ a(c) \neq 0}} (q^{a(c)} -t^{l(c)+1}) \prod_{c\in \mu  } (t^{l(c)} -q^{a(c)+1}) } \\
	                             & \quad \times \frac{(1-t) \prod_{ \substack{ c\in \mu \\ a'(c) = 0}} (1-q^{a'(c)}t^{l'(c)})  }{\prod_{\substack{c\in \mu \\ a(c) = 0}} (q^{a(c)} -t^{l(c)+1}) }.
\end{align*}
The terms in the left factor of the product can be evaluated at $t=1$ to give
\begin{align*}
	 & \left. \frac{ (1-q) B_\mu  \prod_{ \substack{c \in \mu                             \\ a'(c) \neq 0}} (1-q^{a'(c)}t^{l'(c)}) }
	{\prod_{\substack {c \in \mu                                                          \\ a(c) \neq 0}} (q^{a(c)} -t^{l(c)+1}) \prod\limits_{c \in \mu} (t^{l(c)} -q^{a(c)+1}) } \right\rvert_{t=1} \\
	 & \quad = \frac{ (1-q) \sum_{i=1}^{\ell(\mu)} [\mu_i]_q   \prod_{\substack{c \in \mu \\ a'(c) \neq 0}} (1-q^{a'(c)})  }
	{\prod_{\substack {c \in \mu                                                          \\ a(c) \neq 0}} (q^{a(c)} -1) \prod\limits_{c\in \mu  } (1 -q^{a(c)+1}) }. \\
\end{align*}
Now using the fact that
\begin{align*}
	\frac{ \prod_{\substack{ c\in \mu \\ a'(c) \neq 0}} (1-q^{a'(c)}) }
	{ \prod_{\substack {c \in \mu     \\ a(c) \neq 0}} (q^{a(c)} -1) }  = (-1)^{n-\ell(\mu)} && \text{ and } &&
	\prod_{c \in \mu} (1 -q^{a(c)+1}) = (q;q)_\mu,
\end{align*}
we have
\[ \left. \frac{ (1-q) B_\mu \prod_{ \substack{c \in \mu \\ a'(c) \neq 0}} (1-q^{a'(c)}t^{l'(c)}) }
	{\prod_{\substack {c \in \mu \\ a(c) \neq 0}} (q^{a(c)} -t^{l(c)+1}) \prod_{c \in \mu} (t^{l(c)} -q^{a(c)+1}) } \right\rvert_{t=1}
	= (-1)^{n-\ell(\mu)} \frac{1-q}{(q;q)_\mu} \sum_{c \in \mu} [\mu_i]_q. \]
The second term in the product can also be specialized to $t=1$, though some care is needed. Note first that
\begin{align*}
	\frac{(1-t) \prod_{\substack{c \in \mu                                           \\ a'(c) = 0}} (1-q^{a'(c)}t^{l'(c)})  }{\prod_{\substack{c \in \mu \\ a(c) = 0}} (q^{a(c)} -t^{l(c)+1}) }
	 & = \frac{(1-t) \prod_{ \substack{c \in \mu                                     \\ a'(c) = 0}} (1-t^{l'(c)})  }{\prod_{\substack{c \in \mu \\ a(c) = 0}} (1 -t^{l(c)+1}) } \\
	 & =  \frac{(1-t) (t;t)_{\ell(\mu)-1}}{(t;t)_{m_1(\mu)} \cdots (t;t)_{m_n(\mu)}} \\
	 & = \frac{1}{[\ell(\mu)]_t} \qbinom{\ell(\mu)}{m_1(\mu), \dots, m_n(\mu)}_t,
\end{align*}
where $m_i(\mu)$ is the multiplicity of $i$ in $\mu$. This can now be specialized at $t=1$.

We conclude that \[ \left. \frac{M B_\mu \Pi_\mu}{w_\mu} \right\rvert_{t=1} = \frac{(-1)^{n-\ell(\mu)} }{\ell(\mu)}  \binom{\ell(\mu) }{ m_1(\mu),\dots, m_n(\mu)}  (q;q)_\mu^{-1}       (1-q) \sum_{c \in \mu} [\mu_i]_q. \]

There is some further notation and simplifications that will help us interpret this product combinatorially.
First, note that \[ (1-q) \sum_{i=1}^{\ell(\mu)} [\mu_i]_q = \sum_{i=1}^{\ell(\mu)} (1-q ^{\mu_i}), \] and that \[ \binom{\ell(\mu) }{ m_1(\mu),\dots, m_n(\mu)} = \# R(\mu) \] is the number of rearrangements $\alpha=(\alpha_1,\dots, \alpha_{\ell}) \in \R(\mu)$ of the parts of $\mu$.
Therefore, \[ \# R(\mu) \cdot \sum_{i=1}^{\ell(\mu)} (1-q^{\mu_i}) = \sum_{\alpha \in \R(\mu)} \sum_{i=1}^{\ell(\mu)}( 1-q^{\alpha_i}). \]

This corresponds to selecting a rearrangement $\alpha$ of $\mu$ then selecting some $i$ from $1$ to $\ell(\mu)$. Equivalently, we can first select a rearrangement $r$, take $1-q^{\alpha_1}$, then circularly rearrange $\alpha$, keeping this selection of $1-q^{\alpha_1}$. Since there are $\ell(\mu)$ circular rearrangements, we have that \[ \frac{\# R(\mu)}{\ell(\mu)} \sum_{i=1}^{\ell(\mu)} (1-q^{\mu_i}) = \sum_{\alpha \in R(\mu)} (1-q^{\alpha_1}) \] and so we can conclude that \[ \left. \frac{MB_\mu \Pi_\mu}{w_\mu} \right\rvert_{t=1} = (-1)^{n-\ell(\mu)}(q;q)_{\mu}^{-1}\sum_{\alpha \in R(\mu)} (1-q^{\alpha_1}). \]

Now for the second component, we need the classical result (see \cite{Stanley-Book-1999} and \cite{Haglund-Book-2008}) that
\begin{align*}
	(q;q)_n h_n\left[\frac{X}{1-q} \right] & = \sum_{w= (w_1,\dots,w_n) \in \mathbb{N}_+^n} q^{\maj(w)}  x_{w_1} \cdots x_{w_n}     \\
	                                       & = \sum_{w= (w_1,\dots,w_n) \in \mathbb{N}_+^n} q^{\comaj(w)}  x_{w_1} \cdots x_{w_n}   \\
	                                       & = \sum_{w= (w_1,\dots,w_n) \in \mathbb{N}_+^n} q^{\revmaj(w)}  x_{w_1} \cdots x_{w_n}.
\end{align*}
To get the combinatorial objects we want, it will be best to choose the last of these equalities involving the reverse major index.
Thus, if we want the coefficient of the monomial symmetric function in
\[ {(q;q)}_\mu h_\mu \left[ \frac{X}{1-q} \right] = \sum_{\lambda \vdash n} m_\lambda[X] \left\langle  {(q;q)}_\mu h_\mu \left[ \frac{X}{1-q} \right] , h_\lambda \right \rangle \]
we find that we must have
\[ \left\langle {(q;q)}_\mu h_\mu \left[ \frac{X}{1-q} \right] , h_\lambda \right \rangle = \sum_{ \vec{w} \in \WV(\lambda, \mu)} q^{\revmaj(\vec{w})}, \]
where $\revmaj(\vec{w})= \revmaj(w^1)+\cdots+ \revmaj(w^n)$. This means that we have
\[ \left. \langle \Ht_\mu, h_\lambda \rangle \right\rvert_{t=1} = \sum_{\vec{w} \in \WV(\lambda, \mu)} q^{\revmaj(\vec{w})}. \]

For the last term in the summation, we start with the specialization \[ \Ht_\mu[X;q,1] = {(q;q)_\mu} h_\mu \left[ \frac{X}{1-q} \right]: \]
the Cauchy Identity tells us that for any two expressions $X,Y$, and any two dual bases $\{u_\lambda\}_\lambda$, $\{v_\lambda\}$ under the Hall scalar product, we have \[ h_{n}[XY] = \sum_{\lambda \vdash n} u_{\lambda}[X] v_\lambda[X]; \]
in particular, if we use the elementary symmetric functions and forgotten symmetric functions, we have \[ h_{n} \left[ \frac{X}{1-q} \right] = \sum_{\lambda \vdash n} e_\lambda[X] f_\lambda \left[ \frac{1}{1-q} \right], \]
from which we can write \[ \Ht_\mu[X;q,1]  =  {(q;q)_\mu}  \sum_{\eta \vdash n} e_\eta[X] \sum_{ \vec{\nu} \in \PR(\eta,\mu)} f_{\vec{\nu}}\left[ \frac{1}{1-q} \right] \] where $f_{\vec{\nu}} = f_{\nu^1} f_{\nu^2} \cdots f_{\nu^{\ell(\mu)}}$.
We arrive at our preliminary expansion
\begin{align*}
	\left. \Xi e_\lambda \right\rvert_{t=1} & = \sum_{\eta \vdash n} e_\eta \sum_{\mu \vdash n}  (-1)^{n-\ell(\mu)} \sum_{\beta \in \R(\mu)} (1-q^{\beta_1})                                       \\
	                                        & \quad \times \sum_{\vec{w} \in \WV(\lambda, \mu)} q^{\revmaj(\vec{w})} \sum_{\vec{\nu} \in \PR(\eta,\mu)} f_{\vec{\nu}}\left[ \frac{1}{1-q} \right].
\end{align*}

Instead of summing over all $\mu$ then summing over all compositions $\beta \in R(\mu)$ that rearrange to $\mu$, we can instead just some over all compositions $\beta$. Putting everything together, we get the following.

\begin{proposition}
	\label{firstexpansion}
	For any $\lambda \vdash n$, we have
	\[ \left. \Xi e_{\lambda} \right\rvert_{t=1} = \sum_{\eta \vdash n} \sum_{\beta \vDash n} \sum_{\vec{w} \in \WV(\lambda, \beta)} \sum_{\vec{\nu} \in \PR(\eta,\beta)} q^{\revmaj(\vec{w})} (-1)^{n-\ell(\beta)} (1-q^{\beta_1}) f_{\vec{\nu}}\left[ \frac{1}{1-q} \right] e_\eta. \]

	Moreover, for any $\lambda \vdash n$ and $\gamma \vdash m$, we have
	\[ \widetilde{\Delta}_{m_\gamma} \Xi e_{\lambda} = \sum_{\eta \vdash n}  D_{\lambda, \eta}^{\gamma}(q) e_\eta, \]
	where
	\begin{align*}
		D^{\gamma}_{\lambda, \eta} = \sum_{\beta \vDash n} \sum_{\vec{w} \in \WV(\lambda, \beta)} \sum_{\vec{\nu} \in \PR(\eta,\beta)} & q^{\revmaj(\vec{w})} m_\gamma\left[ \sum_i [\beta_i]_q \right]                         \\
		                                                                                                                               & \times (-1)^{n-\ell(\beta)} (1-q^{\beta_1}) f_{\vec{\nu}}\left[ \frac{1}{1-q} \right].
	\end{align*}
\end{proposition}

\section{Forgotten symmetric functions}

For $\mu \vdash n$ of length $\ell$, the combinatorial formula for the forgotten symmetric function $f_\mu$ \cite{Egecioglu-Remmel-Bricks} is given by
\[
	f_\mu\left[ X \right] = (-1)^{n-\ell} \sum_{ \alpha \in \R(\mu)} \sum_{i_1 \leq \cdots \leq i_{\ell}} x_{i_1}^{\alpha_1} \cdots x_{i_\ell}^{a_\ell}. \]
Now, substituting $X = (1-q)^{-1}$, we get the expansion
\begin{equation}
	\label{eq:fmu}
	f_\mu\left[ \frac{1}{1-q}\right] = (-1)^{n-\ell} \sum_{ \alpha \in \R(\mu)} \sum_{0 \leq i_1 \leq \cdots \leq i_{\ell}} \left( q^{i_1} \right)^{\alpha_1} \cdots \left( q^{i_\ell} \right)^{a_\ell}.
\end{equation}
\begin{definition}
	Let $\mu \vdash n$. A \emph{column-composition tableau of type $\mu$} is a pair $C = (\alpha, c)$ where $\alpha \in \R(\mu)$ is a composition that rearranges to $\mu$, and $c = (c_1 \leq c_2 \leq \dots \leq c_n)$ is a sequence such that \[ c_i < c_{i+1} \implies i \in  \{\alpha_1,\alpha_1+\alpha_2, \dots, \alpha_1+\cdots + \alpha_{\ell-1}\}. \]
	We denote by $\CC_\mu$ the set of column-composition tableaux of type $\mu$, and by $\overline{\CC}_\mu$ the subset of those such that $c_1 = 0$. For $C \in \CC_\mu$, we define the \emph{length} of $C$ as $\ell(C) = \lvert \mu \rvert$ and \emph{size} of $C$ as $\lvert C \rvert = c_1 + c_2 + \dots + c_n$. We will write $c_i(C)$ for $c_i$ when we need to specify the column-composition tableau.
\end{definition}

We can depict the elements of $\CC_\mu$ as follows.
\begin{enumerate}
	\item First, draw a row of size $\lvert \mu \rvert$ which we call the base, and then depict the composition $\alpha \in \R(\mu)$ by separating the columns of the base with vertical bars; for instance, when $\mu = (3,2,2,2,1,1,1)$ and $\alpha = (2,1,2,1,2,1,3)$ we draw the base as
	      \begin{align*}
		      \scalebox{.6}{
			      \begin{tikzpicture}[xscale=-1]
				      \partitionfr{12};
				      \draw[line width = .95mm] (-.25,1) -- (12.25,1);
				      \draw[line width = .75mm] (10,-.25) -- (10,1.25);
				      \draw[line width = .75mm] (9,-.25) -- (9,1.25);
				      \draw[line width = .75mm] (7,-.25) -- (7,1.25);
				      \draw[line width = .75mm] (6,-.25) -- (6,1.25);
				      \draw[line width = .75mm] (4,-.25) -- (4,1.25);
				      \draw[line width = .75mm] (3,-.25) -- (3,1.25);
			      \end{tikzpicture}}
	      \end{align*}
	\item Next, draw $c_i$ cells above the $i^{\ith}$ column of the base; in continuing our example, if $c = (0,0,0,1,1,1,1,1,3,3,3,3)$, we draw it as
	      \begin{align*}
		      \scalebox{.6}{
			      \begin{tikzpicture}[xscale=-1]
				      \partitionfr{12,9,4,4}
				      \draw[line width = .95mm] (-.25,1) -- (12.25,1);
				      \draw[line width = .75mm] (10,-.25) -- (10,1.25);
				      \draw[line width = .75mm] (9,-.25) -- (9,2.25);
				      \draw[line width = .75mm] (7,-.25) -- (7,2.25);
				      \draw[line width = .75mm] (6,-.25) -- (6,2.25);
				      \draw[line width = .75mm] (4,-.25) -- (4,4.25);
				      \draw[line width = .75mm] (3,-.25) -- (3,4.25);
			      \end{tikzpicture}}
	      \end{align*}
\end{enumerate}

Let us define the $q$-enumerators
\begin{align*}
	\bC_\mu = \sum_{C \in \CC_\mu} q^{\lvert C \rvert} &  & \text{ and } &  & \overline{\bC}_\mu = \sum_{C \in \overline{\CC}_\mu} q^{\lvert C \rvert},
\end{align*}
which are power series in $q$. Then, by construction we have the following.

\begin{proposition}
	\begin{align*}
		f_{\mu}\left[ \frac{1}{1-q} \right] = (-1)^{|\mu| - \ell(\mu)} \bC_\mu
		 &  &
		\text{ and }
		 &  &
		(1-q^{\lvert \mu \rvert}) f_{\mu}\left[ \frac{1}{1-q} \right] = (-1)^{|\mu| - \ell(\mu)} \overline{\bC}_\mu.
	\end{align*}
\end{proposition}

\begin{proof}
	In Equation \eqref{eq:fmu}, each term in the principal evaluation of $f_\mu$ is given by selecting a rearrangement $\alpha$ of $\mu$, and choosing $i_1\leq \cdots \leq i_{\ell(\mu)}$. This uniquely determines an element $(\alpha,c) \in \CC_\mu$ where the first $\alpha_1$ columns $c_1,\dots, c_{\alpha_1}$ are of size $i_1$, the next $\alpha_2$ columns $c_{\alpha_1+1},\dots, c_{\alpha_1 + \alpha_2}$ are of size $i_2$, and so on. Since then
	\[
		\sum_{i=1}^{|\mu|} c_i = \sum_{j=1}^{\ell(\mu)} \alpha_j i_j,
	\]
	we see that $q^{\lvert (\alpha,c) \rvert}$ equals the term in Equation \eqref{eq:fmu} corresponding to choosing $\alpha$ and $i_1\leq \cdots \leq i_{\ell(\mu)}$.

	The second equality follows from the fact that if $(\alpha,c) \in \CC_\mu$, then so is $(\alpha,c+1^n)$, where $c+1^n = (c_1+1,\dots,c_n+1)$. This defines an injective map, and we have
	\[
		q^{ \lvert \mu \rvert} \bC_\mu = \sum_{ \substack{ C \in \CC_\mu \\ c_1(C) >0 }} q^{\lvert C \rvert}.
	\]
	Therefore
	\[\bC_\mu -q^{\lvert \mu \rvert} \bC_\mu = \sum_{ \substack{ C \in \CC_\mu \\ c_1 = 0}    }q^{\lvert C \rvert},
	\]
	which gives the last equality in the proposition.
\end{proof}

\section{Combinatorial expansions}

We interpret the terms in Proposition \ref{firstexpansion} by labeling a sequence of column-composition tableaux. Recall that we are trying to compute
\begin{align*}
	D^\gamma_{\lambda, \eta}(q) = \sum_{\beta \vDash n} \sum_{\vec{w} \in \WV(\lambda, \beta)} \sum_{\vec{\nu} \in \PR(\eta,\beta)} & q^{\revmaj(\vec{w})} m_\gamma\left[ \sum_i [\beta_i]_q \right]                         \\
	                                                                                                                                & \times (-1)^{n-\ell(\beta)} (1-q^{\beta_1}) f_{\vec{\nu}}\left[ \frac{1}{1-q} \right],
\end{align*}
that is, the coefficient in the expansion $\widetilde{\Delta}_{m_\lambda} \Xi e_\lambda = \sum_{\eta} D^\gamma_{\lambda,\eta}  e_\eta$.

First note that for $\vec{\nu} \in \PR(\eta,\beta)$, we have \[ (-1)^{n-\ell(\beta)} (1-q^{\beta_1}) f_{\vec{\nu}} \left[ \frac{1}{1-q} \right] = (-1)^{\ell(\eta) - \ell(\beta)}\overline{\bC}_{\nu^1} {\bC}_{\nu^2} \cdots  {\bC}_{\nu^{\ell(\beta)}},\]
The sign was computed from the fact that $\ell(\nu^1) + \cdots + \ell(\nu^{\ell(\beta)}) = \ell(\eta)$, and therefore
\[
	(-1)^{n - \ell(\beta)} (-1)^{\ell(\nu^1) - \lvert \nu^1 \rvert}  \cdots (-1)^{\ell(\nu^{\ell(\beta)}) - \lvert \nu^{\ell(\beta)} \rvert}  = (-1)^{\ell(\eta)- \ell(\beta) }\]

\begin{definition}
	A \emph{labeled column-composition tableaux} is a triple $(C, w, l)$, where $C$ is a column-composition tableau, $w \in \mathbb{N}_+^{\ell(C)}$, and $l \in \mathbb{N}^{\ell(C)}$
\end{definition}

\begin{definition}
	\label{LCCTDefinition}

	Let $\lambda, \eta \vdash n$, and $\gamma \vdash m$ such that $\ell(\gamma) \leq n$. A \emph{sequence of labeled column-composition tableaux of type $\lambda, \eta, \gamma$} is a tuple of labeled column-composition tableaux $(C^i, w^i, l^i)_{1 \leq i \leq r}$ such that, for $\beta = (\beta_1, \dots, \beta_r)$, $\beta_i = \ell(C^i)$, we have:

	\begin{enumerate}
		\item $C^1 \in \overline{\CC_{\nu^1}}$ and $C^i \in \CC_{\nu^i}$ for $i>1$, for some $\vec{\nu} \in \PR(\eta, \beta)$;
		\item $\vec{w} = (w^1, \dots, w^r) \in \WV(\lambda, \beta)$;
		\item $\vec{l} = (l^1, \dots, l^r) \in \WV(m(\gamma), \beta)$.
	\end{enumerate}

	In other words, a sequence of labeled column-composition tableaux of type $\lambda, \eta, \gamma$, is a tuple of column-composition tableaux of sizes $\beta_1, \dots, \beta_r$ such that $c_1(C^1) = 0$, so that to each tableau we associate a partition $\nu^i \vdash \beta_i$ and two words $w^i, l^i$ such that $\vec{\nu}$ rearranges to $\eta$, the global content of $\vec{w}$ is given by $\lambda$, and the letters of $l$ are the parts of $\gamma$ followed by an appropriate number of trailing zeros.

	We denote by $\LC_{\lambda,\eta}^\gamma$ the set of sequences of column-composition tableaux of type $\lambda, \eta, \gamma$. For $T = (T_i)_{1 \leq i \leq r} \in \LC_{\lambda,\eta}^\gamma$, we set $w(T_i) = w^i$ and $l(T_i) = l^i$.
\end{definition}

\begin{definition}
	\label{def:weight}
	For $T = (T_i)_{1 \leq i \leq r} \in \LC_{\lambda,\eta}^\gamma$, with $T_i = (C^i, w^i, l^i)$, let $\nu^i$ be the type of $C^i$, let $\beta_i = \lvert \nu^i \rvert$, and let \[ u(l^i) = \sum_{j=1}^{\beta_i} l^i_j \cdot (\beta_i-j). \] We define
	\begin{align*}
		\weight(T_i) = \lvert C^i \rvert + \revmaj(w^i) + u(l^i)
		 &  & \text{ and }
		 &  & \sign(T_i) = (-1)^{\ell(\nu^i)-1}.
	\end{align*}
	Notice that, for every letter in $w^i$ or $l^i$, its contribution to the weight only depends on the letter itself and the number of letters to its right. Also notice that the sign is given by the parity of the number of vertical bars in $C^i$. Finally, we define
	\begin{align*}
		\weight(T) = \sum_{i=1}^r \weight(T_i) &  & \text{ and } &  & \sign(T) = \prod_{i=1}^r \sign(T_i).
	\end{align*}
\end{definition}

\begin{example}
	\label{ex:CCT-computation}
	We are now going through an example in full detail. Let $\lambda=(3,2,2,2)$, $\eta=(3,2,1,1,1,1)$, $\gamma=(4,3,2,2,1)$, so $\lvert \lambda \rvert = \lvert \eta \rvert = 9$ and $\ell(\gamma) = 5 \leq 9$. For our convenience, we add four trailing zeros to $\gamma$, so $\gamma = (4,3,2,2,1,0,0,0,0)$.

	To build an element of $\LC_{\lambda,\eta}^\gamma$, first choose $\beta \vDash 9$ such that some permutation of $\eta$ refines $\beta$, say $\beta = (3,1,5)$. Next, select $\nu^i \vdash \beta_i$, say $\nu^1 = (2,1)$, $\nu^2 = (1)$, $\nu^3 = (3,1,1)$, so that the union of parts is $\eta$. Then, pick $C^1 \in \overline{\CC}_{\nu^1}$ and $C^i \in \CC_{\nu^i}$ for $i > 1$; say for example
	\begin{align*}
		\scalebox{.6}{
		\begin{tikzpicture}[xscale=-1]
				\path[use as bounding box] (0,-1) rectangle (2,3);
				\partitionfr{3,2}
				\draw[line width = .75mm] (-.25,1) -- (3.25,1);
				\draw[line width = .75mm] (2,-.25) -- (2,2.25);
			\end{tikzpicture}} &  &
		\scalebox{.6}{
		\begin{tikzpicture}[xscale=-1]
				\path[use as bounding box] (0,-1) rectangle (1,3);
				\partitionfr{1,1}
				\draw[line width = .75mm] (-.25,1) -- (1.25,1);
			\end{tikzpicture}} &  &
		\scalebox{.6}{
			\begin{tikzpicture}[xscale=-1]
				\path[use as bounding box] (0,-1) rectangle (5,3);
				\partitionfr{5,5,4}
				\draw[line width = .75mm] (-.25,1) -- (5.25,1);
				\draw[line width = .75mm] (1,-.25) -- (1,3.25);
				\draw[line width = .75mm] (4,-.25) -- (4,3.25);
			\end{tikzpicture}}
	\end{align*}
	(note that $c_1(C^1)=0$). Since $\lambda = (3,2,2,2)$, we have $m(\lambda) = (1,1,1,2,2,3,3,4,4)$. Pick any permutation of it and split it into parts of lengths given by the sizes of the parts of $\beta$, say for example $\vec{w} = ((2,1,2),(4),(3,4,1,3,1))$. Write these words into the bases of the tableaux. We get
	\begin{align*}
		\scalebox{.6}{
		\begin{tikzpicture}[xscale=-1]
				\path[use as bounding box] (0,-1) rectangle (2,3);
				\partitionfr{3,2}
				\draw[line width = .75mm] (-.25,1) -- (3.25,1);
				\draw[line width = .75mm] (2,-.25) -- (2,2.25);
				\draw (.5,.5) node {\Large$2$};
				\draw (1.5,.5) node {\Large$1$};
				\draw (2.5,.5) node {\Large$2$};
			\end{tikzpicture}} &  &
		\scalebox{.6}{
		\begin{tikzpicture}[ xscale=-1 ]
				\path[use as bounding box] (0,-1) rectangle (1,3);
				\partitionfr{1,1}
				\draw[line width = .75mm] (-.25,1) -- (1.25,1);
				\draw (.5,.5) node {\Large$4$};
			\end{tikzpicture}} &  &
		\scalebox{.6}{
			\begin{tikzpicture}[xscale=-1]
				\path[use as bounding box] (0,-1) rectangle (5,3);
				\partitionfr{5,5,4}
				\draw[line width = .75mm] (-.25,1) -- (5.25,1);
				\draw[line width = .75mm] (1,-.25) -- (1,3.25);
				\draw[line width = .75mm] (4,-.25) -- (4,3.25);
				\draw (.5,.5) node {\Large$1$};
				\draw (1.5,.5) node {\Large$3$};
				\draw (2.5,.5) node {\Large$1$};
				\draw (3.5,.5) node {\Large$4$};
				\draw (4.5,.5) node {\Large$3$};
			\end{tikzpicture}}
	\end{align*}
	Finally, pick any permutation of the parts of $\gamma$, and again split it into parts of lengths given by the sizes of the parts of $\beta$, say $\vec{l} = ((0,2,0),(1),(2,0,4,3,0))$. Write it underneath the bases of the tableaux. We get
	\begin{align*}
		\scalebox{.6}{
		\begin{tikzpicture}[xscale=-1]
				\path[use as bounding box] (0,-1) rectangle (2,3);
				\partitionfr{3,2}
				\draw[line width = .75mm] (-.25,1) -- (3.25,1);
				\draw[line width = .75mm] (2,-.25) -- (2,2.25);
				\draw (.5,.5) node {\Large$2$};
				\draw (1.5,.5) node {\Large$1$};
				\draw (2.5,.5) node {\Large$2$};
				\draw (.5,-.5) node {\Large$0$};
				\draw (1.5,-.5) node {\Large$2$};
				\draw (2.5,-.5) node {\Large$0$};
			\end{tikzpicture}} &  &
		\scalebox{.6}{
		\begin{tikzpicture}[ xscale=-1 ]
				\path[use as bounding box] (0,-1) rectangle (1,3);
				\partitionfr{1,1}
				\draw[line width = .75mm] (-.25,1) -- (1.25,1);
				\draw (.5,.5) node {\Large$4$};
				\draw (.5,-.5) node {\Large$1$};
			\end{tikzpicture}} &  &
		\scalebox{.6}{
			\begin{tikzpicture}[xscale=-1]
				\path[use as bounding box] (0,-1) rectangle (5,3);
				\partitionfr{5,5,4}
				\draw[line width = .75mm] (-.25,1) -- (5.25,1);
				\draw[line width = .75mm] (1,-.25) -- (1,3.25);
				\draw[line width = .75mm] (4,-.25) -- (4,3.25);
				\draw (.5,.5) node {\Large$1$};
				\draw (1.5,.5) node {\Large$3$};
				\draw (2.5,.5) node {\Large$1$};
				\draw (3.5,.5) node {\Large$4$};
				\draw (4.5,.5) node {\Large$3$};
				\draw (.5,-.5) node {\Large$0$};
				\draw (1.5,-.5) node {\Large$3$};
				\draw (2.5,-.5) node {\Large$4$};
				\draw (3.5,-.5) node {\Large$0$};
				\draw (4.5,-.5) node {\Large$2$};
			\end{tikzpicture}}
	\end{align*}
	which is an element of $\LC_{(3,2,2,2),(3,2,1,1,1,1)}^{(4,3,2,2,1)}$
	We now want to compute the weight and the sign of this sequence. Since there are $3$ vertical bars, we have that the sign is given by $(-1)^3$.

	We can compute the weight of this sequence of labeled column-composition tableaux in three steps.  First count the number of cells above the base rows: there are $2$, $1$, and $5+4=9$ cells respectively, so the total weight given by the cells is $12$.

	The weight corresponding to the labels in the base is found by taking the reverse major index of each individual base. We compute this by taking \[ \sum_{ w^i_j < w^i_{j+1}} \# \{ \text{cells in the same base and on the right of $w^i_j$} \}. \]
	In the above example, we have that $w^1 = (2,1,2)$ has an ascent in position $2$, and there is one cell to its right. Therefore $\revmaj(w^1) = 1$. Since $w^2$ has length $1$, it has no ascents. Finally, $w^3 = (3,4,1,3,1)$ has an ascent in position $1$ and one in position $3$. There are $4$ cells to the right of the label in position $2$, and $2$ cells to the right of the last ascent. Therefore, $\revmaj(w^3) = 4+2$, and the total contribution given by $\vec{w}$ is $7$.

	The last step is to calculate the contribution of labels underneath the base rows. For this, we will say a label $l^i_j$ has $\beta_i-j$ cells on its right, since these are the number of cells in its base row to the right of the label. We take \[ \sum_{i,j} l^i_j \times  \# \{ \text{cells in the same base and on the right of $l^i_j$} \}. \]
	The first nonzero label from the left is a $2$ in the second column of $T_1$. There is one cell to its right, meaning this label contributes by $2 \cdot 1$ to the weight. The next nonzero label is a $1$, but there are no cells to its right, so its contribution is $1 \cdot 0 = 0$. Similarly, $T_2$ has size $1$ so its contribution is also $0$. Finally, in $T_3$, the first label $2$ has $4$ cells to its right (so it contributes $2 \cdot 4$). The next label is a $4$ and it has $2$ cells to its right (so it contributes $4 \cdot 2$). The last nonzero label is a $3$ and it has one cell to its right (so it contributes $3 \cdot 1$). Therefore, the labels under the base rows collectively contribute a factor of $q^{2+8+8+3}$ to the weight. Putting everything together, we have
	\begin{align*}
		\weight(T) = (2+1+9) + (1+0+4+2) + (2+8+8+3) = 40
	\end{align*}
	and $\sign(T) = (-1)^3 = -1$.
\end{example}

The following proposition is an immediate consequence of our construction.

\begin{proposition}
	\[ D^\gamma_{\lambda, \eta} = \sum_{T \in \LC_{\lambda, \eta}^\gamma} q^{\weight(T)} \sign(T). \]
\end{proposition}

\begin{proof}
	Let $n = \lvert \eta \rvert$. The terms in $D^\gamma_{\lambda, \eta} $ are found by first selecting a composition $\beta \vDash n$ such that some permutation of $\eta$ refines $\beta$, a word vector $\vec{w} \in \WV(\lambda, \beta)$, and a partition vector $\vec{\nu} \in \PR(\eta,\beta)$. This gives us the term \[ q^{\revmaj(\vec{w})} m_\gamma\left[ \sum_i [\beta_i]_q \right] (-1)^{n-\ell(\beta)} (1-q^{\beta_1}) f_{\vec{\nu}}\left[ \frac{1}{1-q}\right] \] which we can rewrite as \[ q^{\revmaj(\vec{w})} m_\gamma\left[ \sum_i [\beta_i]_q \right] (-1)^{\ell(\eta)-\ell(\beta)} \overline{\bC}_{\nu^1} \bC_{\nu^2} \cdots \bC_{\nu^{\ell(\beta)}}. \]
	From the combinatorial formula for the monomial symmetric functions, we have
	\[
		m_\gamma\left[  [\beta_1]_q + \cdots + [\beta_{\ell(\beta)}]_q \right] = \sum_{ \vec{l} \in \WV(m(\gamma)),\beta} q^{u(l^1)+\cdots + u (l^{\ell(\beta)})}.
	\]
	Substituting this in for the term, we get
	\[
		\sum_{ \vec{l} \in \WV(m(\gamma),\beta)} (-1)^{\ell(\eta)-\ell(\beta)}  q^{u(l^1)+\cdots + u (l^{\ell(\beta)})} q^{\revmaj(\vec{w})}  \overline{\bC}_{\nu^1} \bC_{\nu^2} \cdots \bC_{\nu^{\ell(\beta)}}
	\]
	We now see that the monomials in this expansion are found by selecting $C^i \in \CC_{\nu^i}$ for each $i$, giving the term
	\begin{align*}
		(-1)^{\ell(\eta)-\ell(\beta)} q^{\revmaj(\vec{w})} & q^{u(l^1)+\cdots + u (l^{\ell(\beta)})} q^{\lvert C^1 \rvert + \cdots + \lvert C^{\ell(\beta)} \rvert} \\ & = \prod_{i=1}^{\ell(\beta)}  (-1)^{\ell(\nu^i)-1} q^{\revmaj(w^i) + u(l^i) + |C^i|}.
	\end{align*}
	This uniquely determines an element $T \in \LC^{\gamma}_{\lambda,\eta}$, with $T_i = (C^i,w^i,l^i)$, and its sign and weight give precisely the same monomial and sign.
\end{proof}

\section{A weight-preserving, sign-reversing involution}

Our goal is to now give a positive expansion for the coefficients $D_{\lambda,\eta}^\gamma (q)$. To achieve this result, we will define a weight-preserving, sign-reversing involution $\psi \colon \LC_{\lambda,\eta}^\gamma \rightarrow  \LC_{\lambda,\eta}^\gamma$ whose fixed points $U^{\gamma}_{\lambda, \eta}$ give \[ D_{\lambda,\eta}^\gamma (q) = \sum_{T \in U_{\lambda,\eta}^\gamma} q^{\weight(T)}. \]

To construct $\psi$, we need to introduce a split map. Suppose $S = (C, w, l)$ is one of the possible labeled column-composition tableaux appearing in a sequence $T \in \LC_{\lambda,\eta}^\gamma$. Let $C = (\alpha, c)$, and recall that this means that the vertical bars appearing in $C$ are in positions $\alpha_1, \alpha_1 + \alpha_2, \dots, \alpha_1 + \dots + \alpha_{\ell-1}$. Let $d = \# \{ 1 \leq i \leq \alpha_1 \mid w_i < w_{i+1} \}$, that is, $d = \asc(w_1, \dots, w_{\alpha_1+1})$ (we don't count the last position if $\alpha$ has just one part).

\begin{definition}
	Suppose that $S$ has at least one bar, i.e.\ $\lvert \alpha \rvert \neq (\alpha_1)$. Then we say that $S$ can split, and define $\spl(S) = (S_1, S_2)$, where $S_1$ is the portion of $S$ occurring before the first vertical bar, and $S_2$ is obtained from the portion of $S$ after the first vertical bar by adding $d + \lvert l^1 \rvert$ cells to each column, $\lvert l^1 \rvert$ being the sum of the labels in $l$ appearing before the first bar (see Example~\ref{ex:split} for a pictorial realization).

	More formally, we have
	\begin{align*}
		S_1 & = ((\alpha_1, (c_1, \dots, c_{\alpha_1})), w^1, l^1),                                                                                              \\
		S_2 & = (((\alpha_2, \dots, \alpha_\ell), (c_{\alpha_1+1} + d + \lvert l^1 \rvert, \dots, c_{\lvert \alpha \rvert} + d + \lvert l^1 \rvert)), w^2, l^2).
	\end{align*}
	where we define $w^1 = (w_1, \dots, w_{\alpha_1})$, $l^1 = (l_1, \dots, l_{\alpha_1})$, and $w^2 = (w_{\alpha_1 + 1}, \dots, w_{\lvert \alpha \rvert})$, $l^2 = (l_{\alpha_1 + 1}, \dots, l_{\lvert \alpha \rvert})$.
\end{definition}

\begin{example}
	\label{ex:split}
	Let $S ((\alpha, c), w, l)$, $\alpha=(3,1,1)$, $c=(1,1,1,1,2)$, $w=(7,4,7,7,5)$, $l=(0,0,2,0,3)$. We split it after $\alpha_1 = 3$ cells. We have $1$ ascent in $(7,4,7,7)$ so $d=1$, and we have $\lvert l^1 \rvert = 0+0+2 = 2$, so we add $3$ cells to each column in $S_2$, and get

	\begin{align}
		\label{eq:separate}
		\begin{tikzpicture}[baseline=1.15cm]
			\draw (0,0) node { };
			\draw (0,1.25) node {$\spl$};
		\end{tikzpicture}
		\scalebox{.6}{
			\begin{tikzpicture}[xscale=-1, baseline=1.15cm]
				\path[use as bounding box] (-.5,0) rectangle (5.5,6);
				\draw (5.5,1.25) node {$\left(\rule{0cm}{2cm}\right.$};
				\partitionfr{5,5,1}
				\draw[line width = .75mm] (-.25,1) -- (5.25,1);
				\draw[line width = .75mm] (1,-.25) -- (1,3.25);
				\draw[line width = .75mm] (2,-.25) -- (2,2.25);
				\draw (.5,.5) node {\Large$5$};
				\draw (1.5,.5) node {\Large$7$};
				\draw (2.5,.5) node {\Large$7$};
				\draw (3.5,.5) node {\Large$4$};
				\draw (4.5,.5) node {\Large$7$};
				\draw (.5,-.5) node {\Large$3$};
				\draw (1.5,-.5) node {\Large$0$};
				\draw (2.5,-.5) node {\Large$2$};
				\draw (3.5,-.5) node {\Large$0$};
				\draw (4.5,-.5) node {\Large$0$};
				\draw (-.5,1.25) node {$\left.\rule{0cm}{2cm}\right)$};
			\end{tikzpicture}
		}
		                                                  &                &
		\begin{tikzpicture}[baseline=1.15cm]
			\draw (0,0) node {};
			\draw (0,1.25) node {$=$};
		\end{tikzpicture}
		                                                  &                &
		\scalebox{.6}{
			\begin{tikzpicture}[xscale=-1, baseline=1.15cm]
				\path[use as bounding box] (0,0) rectangle (3,6);
				\partitionfr{3,3}
				\draw[line width = .75mm] (-.25,1) -- (3.25,1);
				\draw (.5,.5) node {\Large$7$};
				\draw (1.5,.5) node {\Large$4$};
				\draw (2.5,.5) node {\Large$7$};
				\draw ( .5,-.5) node {\Large$2$};
				\draw (1.5,-.5) node {\Large$0$};
				\draw (2.5,-.5) node {\Large$0$};
			\end{tikzpicture}}
		                                                  & \scalebox{2}{
			\begin{tikzpicture}[baseline=1.15cm]
				\draw (0,0) node {};
				\draw (0,1) node {,};
			\end{tikzpicture}
		}
		                                                  & \scalebox{.6}{
		\begin{tikzpicture}[xscale=-1, baseline=1.15cm]
				\path[use as bounding box] (0,0) rectangle (2,6);
				\partitionfr{2,2,2,2,2,1}
				\draw[line width = .75mm] (-.25,1) -- (2.25,1);
				\draw[line width = .75mm] (1,-.25) -- (1,6.25);
				\draw (.5,.5) node {\Large$5$};
				\draw (1.5,.5) node {\Large$7$};
				\draw (.5,-.5) node {\Large$3$};
				\draw (1.5,-.5) node {\Large$0$};
			\end{tikzpicture}} &
		                                                  &                &
	\end{align}
\end{example}

\begin{proposition}
	\label{weightpreserving}
	The map $\spl$ is weight preserving: if $\spl(S) = (S_1, S_2)$, then $\weight(S) = \weight(S_1) + \weight(S_2)$.
\end{proposition}

\begin{proof}
	Suppose $\spl(S) = (S_1, S_2)$. Let $S = (C, w, l)$, $C = (\alpha, c)$ with $\alpha_1 = v$, and let $\ell(C) = n$. Let us denote $S_1 = (C^1, w^1, l^1)$ and $S_2 = (C^2, w^2, l^2)$.

	By definition, the weight has three components, one coming from the total size, one coming from the $\revmaj$ of the word $w$, and one coming from the labels $l$.

	Let $d = \asc(w_1, \dots, w_{v+1})$. By definition of $\spl$, the number of cells above $S_1$ stays the same, while the number of cells above $S_2$ increases by $\ell(C_2) (d + \lvert l^1 \rvert) = (n-v) (d + \lvert l^1 \rvert)$, so the first component of the total weight increases by the same amount.

	By definition of $\revmaj$, we have
	\begin{align*}
		\revmaj(w) & = \sum_{i \in \Asc(w)} (n-i)                                                                  \\
		           & = \sum_{i \in \Asc(w^1)} (n-i) + \sum_{i \in \Asc(w^2)} (n-(v+i)) + \chi(w_v < w_{v+1}) (n-v) \\
		           & = \sum_{i \in \Asc(w^1)} (n-v+v-i)                                                            \\
		           & \quad + \sum_{i \in \Asc(w^2)} (n-v-i) + \chi(w_v < w_{v+1})(n-v)                             \\
		           & = (n-v) \cdot d + \revmaj(w^1) + \revmaj(w^2)
	\end{align*}
	so the second component of the total weight decreases by $(n-v) \cdot d$.

	Finally, by definition of $u$, we have
	\begin{align*}
		u(w) & = \sum_{i=1}^n l_i \cdot (n-i)                                          \\
		     & = \sum_{i=1}^v l_i \cdot (n-i) + \sum_{i=v+1}^n l_i \cdot (n-i)         \\
		     & = \sum_{i=1}^v l_i \cdot (n-v+v-i) + \sum_{i=1}^{n-v} l_i \cdot (n-v-i) \\
		     & = (n-v) \lvert l^1 \rvert + u(l^1) + u(l^2)                             \\
	\end{align*}
	so the third component of the total weight decreases by $ (n-v) \lvert l^1 \rvert$.

	All these changes cancel out and so the weight is preserved, as desired.
\end{proof}

\begin{definition}
	Given two labeled column-composition tableaux $S_1$, $S_2$, we define \[ \asc(S_1; S_2) \coloneqq \asc(w(S_1) w(S_2)_1), \] that is, the number of ascents in the word of $S_1$ followed by the first letter of $S_2$.
\end{definition}

\begin{lemma}
	Let $S_1, S_2$ be two labeled column-composition tableaux. There exists $S$ such that $\spl(S) = (S_1, S_2)$ if and only if
	\begin{equation}
		\label{eq:join}
		c_1(S_2) \geq c_\ell(S_1) + \asc(S_1;S_2) + \lvert l(S_1) \rvert
	\end{equation}
	If such $S$ exists, then it is unique; we say that $S_1$ can join $S_2$ and set $\join(S_1, S_2) = S$.
\end{lemma}

\begin{proof}
	If such $S$ exists, then \eqref{eq:join} holds by construction. Suppose that \eqref{eq:join} holds. Then we can define $S$ as the labeled column composition tableau obtaining by decreasing the size of each column of $S_2$ by $\asc(S_1; S_2) + \lvert l(S_1) \rvert$ and then concatenating it to $S_1$, also concatenating their words. Equation \eqref{eq:join} ensures that the result is still a column-composition tableau.

	It is now immediate that $\spl(S) = (S_1, S_2)$ and that such $S$ is unique.
\end{proof}

The following lemma is crucial to ensure that our sign-reversing, weight-preserving bijection is well defined.

\begin{lemma}
	\label{splitlemma}
	Let $S_1$, $S$ be labeled column-composition tableaux, and let $\spl(S) = (S_2, S_3)$. Then $S_1$ can join $S_2$ if and only if it can join $S$.
\end{lemma}

\begin{proof}
	By construction, $c_1(S_2) = c_1(S)$, so \eqref{eq:join} holds for $S_1$ and $S_2$ if and only if it holds for $S_1$ and $S$.
\end{proof}

We can now define our bijection as follows.

\begin{definition}
	Given $T = (T_1,\dots, T_r) \in \LC_{\lambda,\eta}^\gamma$, define $\psi(T)$ by the following process:
	\begin{enumerate}
		\item if $r=0$, then $\psi(T) = T$;
		\item if $T_1$ can split, then $\psi(T) = (\spl(T_1), T_2,\dots, T_r)$;
		\item if $T_1$ cannot split and $T_1$ can join $T_2$, then $\psi(T) = (\join(T_1, T_2), T_3,\dots, T_r)$;
		\item otherwise we inductively define $\psi(T) = (T_1, \psi(T_2, \dots, T_r))$.
	\end{enumerate}
\end{definition}

\begin{theorem}
	\label{fixedpointtheorem}
	Let
	\begin{align*}
		U_{\lambda,\eta}^\gamma = \Big\{ T \in \LC_{\lambda,\eta}^\gamma \mid T & \text{ has no vertical bars, and for all $i$}                                         \\
		                                                                        & c_1(T_{i+1}) < c_{\ell(T_i)} (T_i) + \asc(T_i;T_{i+1}) + \lvert l(T_i) \rvert \Big\}.
	\end{align*}
	Then \[ D_{\lambda,\eta}^\gamma (q) = \sum_{T \in U_{\lambda,\eta}} q^{\weight(T)}. \]
\end{theorem}

\begin{proof}
	Since we are using the split map and its inverse, Proposition~\ref{weightpreserving} ensures that $\psi$ is weight-preserving. Furthermore, $\spl$ and $\join$ remove or adds a single vertical bar, so $\psi$ it is sign-reversing.

	We have to make sure that $\psi$ is an involution. Let $T = (T_1, \dots, T_r) \in \LC_{\lambda,\eta}^\gamma$. If $\psi(T) = T$, then clearly $\psi^2(T) = T$.

	Suppose that $\psi(T) = (T_1, \dots, T_{i-1}, \spl(T_i), T_{i+1}, \dots, T_r)$. Let $\spl(T_i) = (S_1, S_2)$. By construction, $T_1, \dots, T_{i-1}$ cannot split, and $T_j$ cannot join $T_{j+1}$ for $j < i$. By Lemma~\ref{splitlemma}, since $T_{i-1}$ cannot join $T_i$, it also cannot join $S_1$. By construction, $S_1$ and $S_2$ can join, so $\psi^2(T) = T$.

	Suppose instead that $\psi(T) = (T_1, \dots, T_{i-1}, \join(T_i, T_{i+1}), T_{i+2}, \dots, T_r)$. By construction, $T_1, \dots, T_{i-1}$ cannot split, and $T_j$ cannot join $T_{j+1}$ for $j < i$.	By Lemma~\ref{splitlemma}, since $T_{i-1}$ cannot join $T_i$, it also cannot join $\join(T_i, T_{i+1})$. By construction, $\join(T_i, T_{i+1})$ can split, so again $\psi^2(T) = T$ and $\psi$ is an involution.

	The set of fixed points is the set of labeled column composition whose parts cannot split or be joined, and the conditions for that to hold are exactly the conditions given in the definition of $U^{\gamma}_{\lambda,\eta}$.
\end{proof}

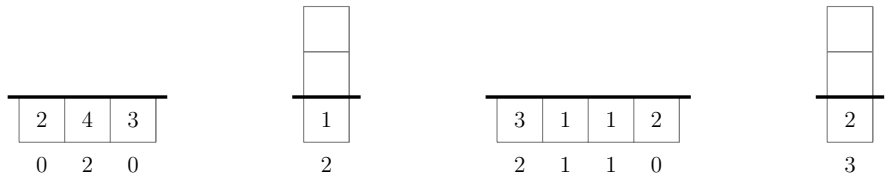
\begin{figure}[!ht]
	\centering
	\begin{align*}
		\scalebox{.6}{
		\begin{tikzpicture}[ xscale=-1]
				\partitionfr{3}
				\draw[line width = .75mm] (-.25,1) -- (3.25,1);
				\draw (.5,.5) node {\Large$3$};
				\draw (1.5,.5) node {\Large$4$};
				\draw (2.5,.5) node {\Large$2$};
				\draw (.5,-.5) node {\Large$0$};
				\draw (1.5,-.5) node {\Large$2$};
				\draw (2.5,-.5) node {\Large$0$};
			\end{tikzpicture}} &  &
		\scalebox{.6}{
		\begin{tikzpicture}[ xscale=-1]
				\partitionfr{1,1,1}
				\draw[line width = .75mm] (-.25,1) -- (1.25,1);
				\draw (.5,.5) node {\Large$1$};
				\draw (.5,-.5) node {\Large$2$};
			\end{tikzpicture}} &  &
		\scalebox{.6}{
		\begin{tikzpicture}[ xscale=-1]
				\partitionfr{4}
				\draw[line width = .75mm] (-.25,1) -- (4.25,1);
				\draw (.5,.5) node {\Large$2$};
				\draw (1.5,.5) node {\Large$1$};
				\draw (2.5,.5) node {\Large$1$};
				\draw (3.5,.5) node {\Large$3$};
				\draw (.5,-.5) node {\Large$0$};
				\draw (1.5,-.5) node {\Large$1$};
				\draw (2.5,-.5) node {\Large$1$};
				\draw (3.5,-.5) node {\Large$2$};
			\end{tikzpicture}} &  &
		\scalebox{.6}{
			\begin{tikzpicture}[ xscale=-1]
				\partitionfr{1,1,1}
				\draw[line width = .75mm] (-.25,1) -- (1.25,1);
				\draw (.5,.5) node {\Large$2$};
				\draw (.5,-.5) node {\Large$3$};
			\end{tikzpicture}}
	\end{align*}
	\caption{A fixed point of $\psi$}
	\label{fig:fixedpoint}
\end{figure}

\begin{example}
	We can read the type of the element in Figure~\ref{fig:fixedpoint} as follows. Since the rows have lengths $3,1,4,1$ respectively, we know that $\eta = (4,3,1,1)$. Since the words in the base rows are $(2,4,3),(1),(3,1,1,2),(2)$, which have multiplicities given by $1^3 2^3 3^1 4^1$, we have $\lambda = (3,3,1,1).$ Lastly, the labels underneath the rows rearrange to $(3,2,2,2,1,1,0,0,0)$, meaning $\gamma = (3,2,2,2,1,1).$
\end{example}

\section{A bijection to ascent polyominoes}

In this section we define an intermediate family of objects that will turn out handy to describe our bijection between the fixed points of $\psi$ and $\gamma$-parking functions, namely \emph{ascent labeled polyominoes}.

\begin{definition}
	\label{LPPDefinition}
	An $m \times n$ ascent labeled parallelogram polyomino is a triple $(P,Q,w)$ such that $(P,Q)$ is an $m \times n$ parallelogram polyomino (as in Definition~\ref{def:polyomino}), and $w \in \mathbb{N}^n$ is such that if $Q$ has no East steps on the line $y = i-1$ (or has only $1$ East step if $i=1$, since the first step of $Q$ must be East), then $w_i \geq w_{i+1}$.
\end{definition}

\begin{definition}
	\label{def:ascent-polyominoes}
	Let $(P, Q, w)$ be an ascent labeled parallelogram polyomino. Let $\lambda \vDash n$ be the content of $w$ (that is, $m(w) = 1^{\lambda_1} 2^{\lambda_2}\cdots$), and let $\eta \vdash n$ be the partition whose block sizes are the lengths of the maximal streaks of North steps in $P$, in some order. Let \[ \beta_i \coloneqq \# \{ \text{East steps of $Q$ on the line $y = i-1$} \} - \chi(i=1) - \chi(i \in \Asc(w)), \] and let $\gamma$ be the partition obtained by rearranging $(\beta_1, \dots, \beta_n)$ and removing zeros.

	We define $\type(P, Q, w) = (\lambda, \eta, \gamma)$, and call $\PP_{\lambda, \eta}^\gamma$ the set of ascent labeled polyominoes of type $(\lambda, \eta, \gamma)$. Note that the height is fixed by the type but the width isn't, as it depends on the number of ascents in $w$.
\end{definition}

We will now give a bijection $\varphi \colon U^{\gamma}_{\lambda,\eta} \rightarrow \PP^\gamma_{\lambda,\eta}$ from the set of fixed points of $\psi$ of given type, and ascent labeled parallelogram polyominoes of the same type.

In order to describe the bijection, for $T = (T_1, \dots, T_r)\in U^{\gamma}_{\lambda,\eta}$, we need to define a triple $\varphi(T) = (P(T),Q(T),w(T))$ corresponding to the polyomino and its labels.

\begin{definition}
	Let $T = (T_1, \dots, T_r)\in U^{\gamma}_{\lambda,\eta}$, with $T_i = (C^i, w^i, l^i)$. First, we define $w(T) \coloneqq w^1 \cdots w^r$ (the concatenation).

	Next, let $l(T) = l^1 \cdots l^r$ and let $r_i(T) = l(T)_i + \chi(i \in \Asc(w))$. We define \[ Q(T) \coloneqq E E^{r_1(T)} N E^{r_2(T)} N \cdots E^{r_n(T)} N, \] that is, the path with $r_i(T)$ East step on the line $y=i-1$, plus $1$ if $i=1$.

	Finally, let $s_i(T) = c_{\ell(T_i)}(T_i) + \asc(T_i; T_{i+1}) + \lvert l^i \rvert - c_1(T_{i+1})$, which is guaranteed to be positive by the fact that $T$ is a fixed point of $\psi$. We set \[ P(T) \coloneqq N^{\ell(T_1)} E^{s_1(T)} N^{\ell(T_2)} E^{s_2(T)} \cdots N^{\ell(T_n)} E^{s_n(T)} E. \]
\end{definition}

\begin{figure}[!ht]
	\centering
	\begin{align*}
		\scalebox{.6}{
			\begin{tikzpicture}[scale=1]
				\dyckpathmn{0,0}{15}{9}{0};
				\draw[line width=3pt,red] (-.2,0) -- (-.2,3)  ;
				\draw[line width=3pt,red] (.8,3) -- (.8,4)  ;
				\draw[line width=3pt,red] (5.8,4) -- (5.8,8) ;
				\draw[line width=3pt,red] (8.8,8) -- (8.8,9) ;
				\draw[line width=3pt,green] (1,-.2) -- (2,-.2)  ;
				\draw[line width=3pt,green] (2 ,1-.2) -- (4,1-.2)   ;
				\draw[line width=3pt,green] (4,3-.2) -- (7,3-.2)  ;
				\draw[line width=3pt,green] (7,4-.2) -- (9,4-.2)   ;
				\draw[line width=3pt,green] (9,5-.2) -- (10,5-.2)   ;
				\draw[line width=3pt,green] (10,6-.2) -- (12,6-.2)  ;
				\draw[line width=3pt,green] (12,8-.2) -- (15,8-.2)  ;
				\dyckpathmnnogrid{0,0}{15}{9}{1,1,1,0,1,0,0,0,0,0,1,1,1,1,0,0,0,1,0,0,0,0,0,0};
				\draw (.5,.5) node {\LARGE$2$};
				\draw (.5,1.5) node {\LARGE$4$};
				\draw (.5,2.5) node {\LARGE$3$};
				\draw (1.5,3.5) node {\LARGE$1$};
				\draw (6.5,4.5) node {\LARGE$3$};
				\draw (6.5,5.5) node {\LARGE$1$};
				\draw (6.5,6.5) node {\LARGE$1$};
				\draw (6.5,7.5) node {\LARGE$2$};
				\draw (9.5,8.5) node {\LARGE$2$};
				\dyckpathmnnogrid{0,0}{15}{9}{0,0,1,0,0,1,1,0,0,0,1,0,0,1,0,1,0,0,1,1,0,0,0,1};
			\end{tikzpicture}}
	\end{align*}
	\caption{The image $\varphi(T)$ of the fixed point in Figure~\ref{fig:fixedpoint}}
	\label{fig:polyomino}
\end{figure}
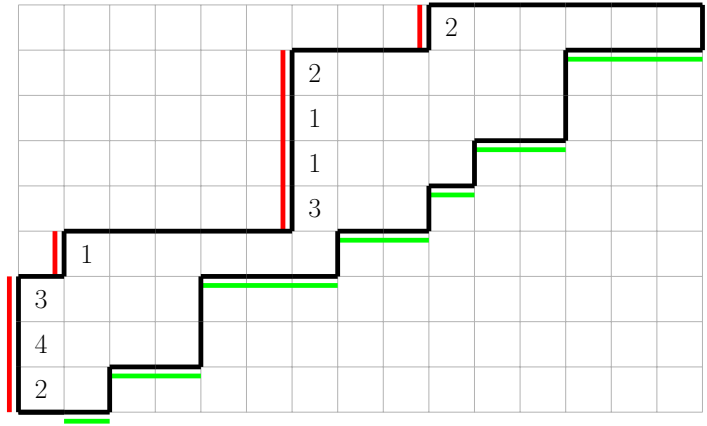

\begin{example}
	We demonstrate the bijection for the sequence of labeled column composition tableaux $T \in U^{(3,2,2,2,1)}_{(3,3,1,1),(4,3,1,1)}$ appearing in Figure~\ref{fig:fixedpoint}. We have  $w = w(T) = (2,4,3,1,3,1,1,2,2)$ and $l = l(T) = (0,2,0,2,2,1,1,0,3)$.

	We have $\Asc(w) = \{1, 4, 7\}$, so we get \[ Q(T) = E E^{0+1} N E^{2+0} N E^{0+0} N E^{2+1} N E^{2+0} N E^{1+0} N E^{1+1} N E^{0+0} N E^{3+0} N. \]

	Finally, we have $s_1(T) = 0 + 1 + 2 - 2 = 1$, $s_2(T) = 2 + 1 + 2 - 0 = 5$, $s_3(T) = 0 + 1 + 4 - 2 = 3$, and $s_4(T) = 2 + 0 + 3 - 0 = 5$, so we end up with \[ P(T) = N^3 E^1 N^1 E^5 N^4 E^3 N^1 E^5 E. \]

	We refer to Figure~\ref{fig:polyomino} for the image of $T$ under $\varphi$. The North steps of $P(T)$ are labeled with the word $w$ written from bottom to top. The North segments are also distinguished in our picture with a red line on its left. We see that the vertical segments of $P(T)$ have lengths $3,1,4,1$, which rearranges to $\eta = (4,3,1,1)$. The green segments highlight the horizontal segments of $Q(T)$ (ignoring the first, mandatory East step), and along each horizontal line, we have lengths $(1,2,0,3,2,1,2,0,3)$. Since $\Asc(w) = \{1,4,7\}$, we subtract term by term to see that \[ (1,2,0,3,2,1,2,0,3) - (1,0,0,1,0,0,1,0,0)= (0,2,0,2,2,1,1,0,3) \] has nonzero parts that rearrange to $\gamma =  (3,2,2,2,1,1)$, as expected.
\end{example}

\begin{figure}[!ht]
	\centering
	\begin{center}\begin{tabular}{ccc}
			\scalebox{.4}{
				\begin{tikzpicture}
					\dyckpathmn{0,0}{7}{9}{1,1,1,0,1,0,1,1,1,1,0,0,1,0,0,0};
					\dyckpathmnnogrid{0,0}{7}{9}{0,1,0,0,1,0,1,1,0,1,1,0,1,1,0,1};
					\draw (.5,.5) node {\Large$2$};
					\draw (.5,1.5) node {\Large$2$};
					\draw (.5,2.5) node {\Large$2$};
					\draw (1.5,3.5) node {\Large$1$};
					\draw (2.5,4.5) node {\Large$1$};
					\draw (2.5,5.5) node {\Large$3$};
					\draw (2.5,6.5) node {\Large$1$};
					\draw (2.5,7.5) node {\Large$2$};
					\draw (4.5,8.5) node {\Large$1$};
				\end{tikzpicture}
			} &
			\scalebox{.4}{
				\begin{tikzpicture}
					\dyckpathmn{0,0}{7}{9}{1,1,1,0,1,0,1,1,1,1,0,0,1,0,0,0};
					\dyckpathmnnogrid{0,0}{7}{9}{0,1,0,0,1,0,1,1,0,1,1,0,1,1,0,1};
					\draw (.5,.5) node {\Large$2$};
					\draw (.5,1.5) node {\Large$2$};
					\draw (.5,2.5) node {\Large$2$};
					\draw (3.5,3.5) node {\Large$1$};
					\draw (3.5,4.5) node {\Large$1$};
					\draw (3.5,5.5) node {\Large$3$};
					\draw (3.5,6.5) node {\Large$1$};
					\draw (3.5,7.5) node {\Large$2$};
					\draw (5.5,8.5) node {\Large$1$};
					\filldraw[line width=3pt,red, fill=red, fill opacity = 0.1] (0,0) -- (1,0) -- (1,3) -- (0,3) --cycle;
					\filldraw[line width=3pt,red, fill=red, fill opacity = 0.1] (1,3) -- (4,3) -- (4,4) -- (1,4) --cycle;
					\filldraw[line width=3pt,red, fill=red, fill opacity = 0.1] (2,4) -- (4,4) -- (4,8) -- (2,8) --cycle;
					\filldraw[line width=3pt,red, fill=red, fill opacity = 0.1] (4,8) -- (6,8) -- (6,9) -- (4,9) --cycle;
				\end{tikzpicture}}
			  &
			\scalebox{.4}{
				\begin{tikzpicture}[rotate=-90]
					\dyckpathmn{0,0}{7}{9}{1,1,1,0,1,0,1,1,1,1,0,0,1,0,0,0};
					\dyckpathmnnogrid{0,0}{7}{9}{0,1,0,0,1,0,1,1,0,1,1,0,1,1 ,0,1};
					\draw (.5,.5) node {\Large$2$};
					\draw (.5,1.5) node {\Large$2$};
					\draw (.5,2.5) node {\Large$2$};
					\draw (3.5,3.5) node {\Large$1$};
					\draw (3.5,4.5) node {\Large$1$};
					\draw (3.5,5.5) node {\Large$3$};
					\draw (3.5,6.5) node {\Large$1$};
					\draw (3.5,7.5) node {\Large$2$};
					\draw (5.5,8.5) node {\Large$1$};
					\filldraw[line width=3pt,red, fill=red, fill opacity = 0.1] (0,0) -- (1,0) -- (1,3) -- (0,3) --cycle;
					\filldraw[line width=3pt,red, fill=red, fill opacity = 0.1] (1,3) -- (4,3) -- (4,4) -- (1,4) --cycle;
					\filldraw[line width=3pt,red, fill=red, fill opacity = 0.1] (2,4) -- (4,4) -- (4,8) -- (2,8) --cycle;
					\filldraw[line width=3pt,red, fill=red, fill opacity = 0.1] (4,8) -- (6,8) -- (6,9) -- (4,9) --cycle;
				\end{tikzpicture}}
			\\1.&2.&3.\\
		\end{tabular}\end{center}
	\begin{center}
		\begin{align*}
			\scalebox{.6}{
			\begin{tikzpicture}[ xscale=-1]
					\partitionfr{3}
					\draw[line width = .75mm] (-.25,1) -- (3.25,1);
					\draw (.5,.5) node {\Large$2$};
					\draw (1.5,.5) node {\Large$2$};
					\draw (2.5,.5) node {\Large$2$};
					\draw (.5,-.5) node {\Large$1$};
					\draw (1.5,-.5) node {\Large$2$};
					\draw (2.5,-.5) node {\Large$0$};
				\end{tikzpicture}} &  &
			\scalebox{.6}{
			\begin{tikzpicture}[ xscale=-1]
					\partitionfr{1,1,1}
					\draw[line width = .75mm] (-.25,1) -- (1.25,1);
					\draw (.5,.5) node {\Large$1$};
					\draw (.5,-.5) node {\Large$0$};
				\end{tikzpicture}} &  &
			\scalebox{.6}{
			\begin{tikzpicture}[ xscale=-1]
					\partitionfr{4,4}
					\draw[line width = .75mm] (-.25,1) -- (4.25,1);
					\draw (.5,.5) node {\Large$2$};
					\draw (1.5,.5) node {\Large$1$};
					\draw (2.5,.5) node {\Large$3$};
					\draw (3.5,.5) node {\Large$1$};
					\draw (.5,-.5) node {\Large$0$};
					\draw (1.5,-.5) node {\Large$0$};
					\draw (2.5,-.5) node {\Large$0$};
					\draw (3.5,-.5) node {\Large$0$};
				\end{tikzpicture}} &  &
			\scalebox{.6}{
				\begin{tikzpicture}[ xscale=-1]
					\partitionfr{1,1}
					\draw[line width = .75mm] (-.25,1) -- (1.25,1);
					\draw (.5,.5) node {\Large$1$};
					\draw (.5,-.5) node {\Large$1$};
				\end{tikzpicture}}
		\end{align*}
		\\4.
	\end{center}
	\caption{A pictorial description of $\varphi^{-1}$}
	\label{fig:CCT-to-ascent-polyomino}
\end{figure}
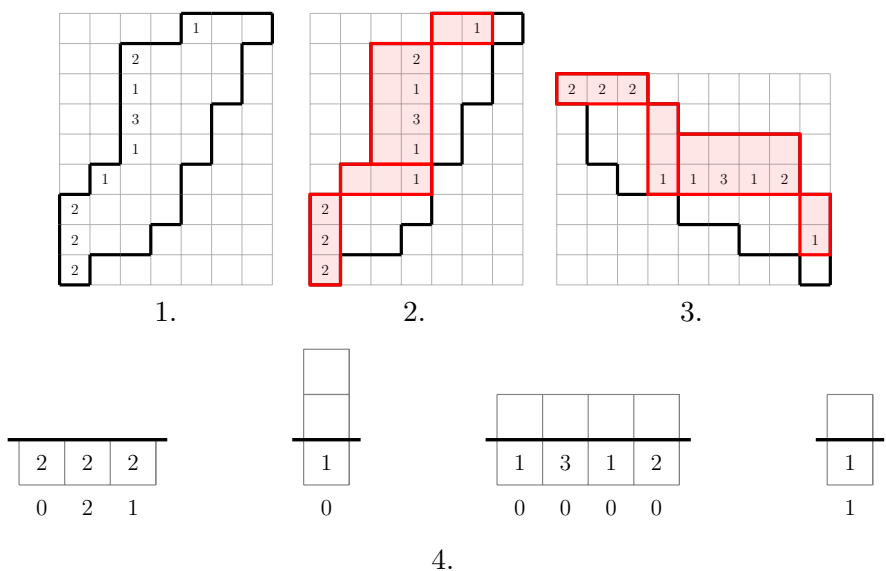

\begin{theorem}
	\label{bijectiontheorem}
	The map $\varphi(T) = (P(T),Q(T),w(T))$ is a bijection between $U_{\lambda,\eta}^\gamma$ and $\PP_{\lambda,\eta}^\gamma$ such that $\weight(T) = \area(\varphi(T))$, that is, it is weight-preserving.
\end{theorem}

\begin{proof}
	We describe the inverse $\varphi^{-1}$ instead. Starting with $S = (P,Q,w) \in \PP_{\lambda, \eta}^\gamma$, for each maximal vertical segment in $P$, draw the maximal rectangle contained in $S$ that has that streak as one of the sides and whose perimeter does not contain any East step in $Q$. Then slide all of the labels in these vertical segments to the opposite side of the maximal rectangle (see Step 2 of Figure~\ref{fig:CCT-to-ascent-polyomino}).

	To construct $(T_1,\dots, T_{\ell(\lambda)})$, we can determine the tableaux by looking at the rectangles: if $T_i = (C^i, w^i, l^i)$, then $C^i$ is the column-composition tableau obtained by rotating the rectangle delimited by the $i^{\ith}$ vertical segment in $S$; $w^i$ is the sequence of labels appearing in the rectangle, read from bottom to top; if $l = l^1 \cdots l^{\ell(\lambda)}$ then $l_j + \chi(j \in \Asc(w))$ is the number of cells of $S$ in the $j^{\ith}$ row that are outside the maximal rectangle and whose bottom segment is an East step of $Q$. This is better seen by rotating $S$ by $90$ degrees clockwise (as in Figure~\ref{fig:CCT-to-ascent-polyomino}, Step 3); let us call this rotated picture $S'$.

	We should note that since $(P,Q,w)$ gives a parallelogram polyomino, the $i+1^{\ith}$ vertical segment in $P$ occurs strictly right of the $i^{\ith}$ vertical segment and also strictly left of $Q$. Using the translation to $T = \varphi^{-1}(S)$, this is equivalent to say that \[ c_1(T_{i+1}) < c_{\ell(T_i)}(T_{i}) + \sum_{j =1}^{\ell(T_i)} \chi(j \in \Asc(w^i w^{i+1})) + \sum_{j=1}^{\ell(T_i)} l_i \] or rather \[ c_1(T_{i+1}) < c_{\ell(T_i)}(T_{i}) + \asc(T_i;T_{i+1}) + \lvert l(T_i)\rvert. \]
	Since the defining property making $P$ a path above $Q$ with respect to $w$ converts directly to the defining relation for elements $T \in U_{\lambda,\eta}^\gamma$, then $\varphi$ is a bijection.

	\begin{figure}[!ht]
		\centering
		\begin{align*}
			\scalebox{.6}{
				\begin{tikzpicture}[rotate=-90]
					\dyckpathmn{0,0}{7}{9}{1,1,1,0,1,0,1,1,1,1,0,0,1,0,0,0};
					\dyckpathmnnogrid{0,0}{7}{9}{0,1,0,0,1,0,1,1,0,1,1,0,1,1,0,1};
					\draw (.5,.5) node {\Large$2$};
					\draw (.5,1.5) node {\Large$2$};
					\draw (.5,2.5) node {\Large$2$};
					\draw (3.5,3.5) node {\Large$1$};
					\draw (3.5,4.5) node {\Large$1$};
					\draw (3.5,5.5) node {\Large$3$};
					\draw (3.5,6.5) node {\Large$1$};
					\draw (3.5,7.5) node {\Large$2$};
					\draw (5.5,8.5) node {\Large$1$};
					\draw[line width=3pt,dashed,blue] (.5,.5) -- (.5,2.5)-- (3.5,2.5)--(3.5,7.5)--(5.5,7.5)--(5.5,8.5)--(6.5,8.5);
					\filldraw[line width=3pt,red, fill=red, fill opacity = 0.1] (0,0) -- (1,0) -- (1,3) -- (0,3) --cycle;
					\filldraw[line width=3pt,red, fill=red, fill opacity = 0.1] (1,3) -- (4,3) -- (4,4) -- (1,4) --cycle;
					\filldraw[line width=3pt,red, fill=red, fill opacity = 0.1] (2,4) -- (4,4) -- (4,8) -- (2,8) --cycle;
					\filldraw[line width=3pt,red, fill=red, fill opacity = 0.1] (4,8) -- (6,8) -- (6,9) -- (4,9) --cycle;
				\end{tikzpicture}}
		\end{align*}
		\caption{$\varphi$ is weight-preserving}
		\label{fig:weight-preserving}
	\end{figure}

	We are now going to show that $\varphi$ is weight-preserving. First, draw a path along the cells of $S'$ starting from the top left cell, and moving East if there are labels directly East, and moves South otherwise (see Figure~\ref{fig:weight-preserving}). Now the area can be computed by counting the number of unit cells with no dashed line through it, those cells covering the minimal area we remove as normalization. The area above the dashed line equals the weight contribution in the $T_i$ given by the cells above the base row.

	The cells below the dashed line will be counted in the following way: In $S'$, directly South of the label $w_j$, there are by construction $l_j + \chi(j \in \Asc(w))$ cells which are adjacent to the path on the left. Each of these cells have the same number of cells weakly East of them, namely the number of cells between $w_j$ and the next base row. The number of cells in these rows is then \[ (l_j + \chi(j \in \Asc(w))) \times \# \{ \text{cells in the same base and on the right of $w_j$} \}. \] Taking the sum over all $j$, we get the number of cells below the dashed line. But now the second and the third factor contributing to $\weight(T)$ are exactly \[ \sum_{j \in \Asc(w)} \# \{ \text{cells in the same base and on the right of $w_j$} \} \] and \[ \sum_j l_j \times \# \{ \text{cells in the same base and on the right of $w_j$} \} \] (as we described in Example~\ref{ex:CCT-computation}), so indeed $\weight(T) = \area(\varphi(T))$, as desired.
\end{proof}

In the end, we get the following result.

\begin{theorem}
	\[ \left. \Delta_{m_\gamma} \Xi e_\lambda \right\rvert_{t=1} =\sum_{\eta \vdash n } e_\eta \sum_{S \in \PP^{\gamma}_{\lambda,\eta} } q^{\area(S)}. \]
\end{theorem}

\section{\texorpdfstring{$\gamma$}{gamma}-Parking Functions}

In this section we complete the construction by giving a weight-preserving bijection between $\PP_{\lambda,\eta}^\gamma$ and the set \[ \PF_{\lambda, \eta}^\gamma \coloneqq \{ p \in \PF_{\lambda}^\gamma \mid \eta(p) = \eta \} \] of $\gamma$-parking function with content $\lambda$ and $e$-composition $\eta$, part of which we have already seen in Definition~\ref{def:e-composition}.

It is slightly easier to describe the inverse, so that is what we will do.

\begin{definition}
	\label{def:iota}
	Let $p = (P, Q, w)$ be a labeled $\gamma$-Dyck path. We define $\iota(p) = (\overline{P}, \overline{Q}, w)$, where $\overline{P}$ and $\overline{Q}$ are the paths obtained from $P$ and $Q$ respectively by removing, for every $i \not \in \Asc(w)$, the first East step of $P$ on the line $y=i$ and the first East step of $Q$ on the line $y=i-1$.
\end{definition}

See Figure~\ref{fig:iota} for an example.

\begin{theorem}
	The map $\iota$ defines a bijection between $\PF_{\lambda, \eta}^\gamma$ and $\PP_{\lambda,\eta}^\gamma$ such that $\area(p) = \area(\iota(p))$.
\end{theorem}

\begin{proof}
	First, notice that $i \not \in \Asc(w)$ means that $w_i \geq w_{i+1}$, and so there must be an East step of $P$ on the line $y=i$ as the labeling is strictly increasing along columns; then, notice that by definition we have at least an East step of $Q$ on each line; finally, since the first East step of $P$ on the line $y = i$ must necessarily be strictly on the left of the first East step of $Q$ on the line $y=i-1$, then $(\overline{P}, \overline{Q})$ is still a parallelogram polyomino.

	Now, since $Q$ has at least one East step on each line (two on the line $y=0$), and we are only removing steps on lines $y=i-1$ where $i \not \in \Asc(w)$, then $\overline{Q}$ is guaranteed to have at least one East step on each line $y=i-1$ where $i \in \Asc(w)$ (two on the line $y=0$ if $1 \in \Asc(w)$). This means that $\iota(p)$ is an ascent labeled polyomino.

	By construction, the word is the same and so the content is also the same; by Definition~\ref{def:e-composition}, the $e$-composition of $p$ is exactly given by the lenghts of the maximal vertical segments of $\overline{P}$, and by Definition~\ref{def:ascent-polyominoes} the last component of the type must be $\gamma$. It follows that $\iota(p) \in \PP_{\lambda,\eta}^\gamma$.

	Since the algorithm that defines $\iota$ is invertible, the map is bijective. Finally, notice that the number of cells in the $i^{\ith}$ row decreases by $1$ if $i \not \in \Asc(w)$ and stays the same if $i \in \Asc(w)$, so we lose $\ell(w) - \asc(w)$ cells; on the other hand, the width of the polyomino also decreases by $\ell(w) - \asc(w)$, so the area stays the same.
\end{proof}

\begin{figure}[!ht]
	\centering
	\begin{align*}
		\scalebox{.6}{
			\begin{tikzpicture}[scale=1]
				\draw[gray!60, thin] (0,0) grid (9,8) (0,0) -- (8,8);
				\filldraw[yellow, opacity=0.3] (0,0) -- (1,0) -- (2,0) -- (2,1) -- (3,1) -- (3,2) -- (4,2) -- (4,3) -- (5,3) -- (5,4) -- (6,4) -- (6,5) -- (7,5) -- (7,6) -- (8,6) -- (8,7) -- (9,7) -- (9,8) -- (8,8) -- (7,8) -- (6,8) -- (6,7) -- (5,7) -- (5,6) -- (4,6) -- (3,6) -- (3,5) -- (3,4) -- (2,4) -- (1,4) -- (1,3) -- (0,3) -- (0,2) -- (0,1) -- (0,0);
				\draw[red, sharp <-sharp >, sharp angle = -45, line width=1.6pt] (0,0) -- (0,1) -- (0,2) -- (0,3) -- (1,3) -- (1,4) -- (2,4) -- (3,4) -- (3,5) -- (3,6) -- (4,6) -- (5,6) -- (5,7) -- (6,7) -- (6,8) -- (7,8) -- (8,8) -- (9,8);
				\draw[green, sharp <-sharp >, sharp angle = 45, line width=1.6pt] (0,0) -- (1,0) -- (2,0) -- (2,1) -- (3,1) -- (3,2) -- (4,2) -- (4,3) -- (5,3) -- (5,4) -- (6,4) -- (6,5) -- (7,5) -- (7,6) -- (8,6) -- (8,7) -- (9,7) -- (9,8);
				\node at (0.5,0.5) {\LARGE$2$};
				\node at (0.5,1.5) {\LARGE$4$};
				\node at (0.5,2.5) {\LARGE$7$};
				\node at (1.5,3.5) {\LARGE$5$};
				\node at (3.5,4.5) {\LARGE$1$};
				\node at (3.5,5.5) {\LARGE$3$};
				\node at (5.5,6.5) {\LARGE$8$};
				\node at (6.5,7.5) {\LARGE$6$};
				\node at (0.5, 3) {\Huge $\times$};
				\node at (1.5, 4) {\Huge $\times$};
				\node at (5.5, 7) {\Huge $\times$};
				\node at (6.5, 8) {\Huge $\times$};
				\node at (3.5, 2) {\Huge $\times$};
				\node at (4.5, 3) {\Huge $\times$};
				\node at (7.5, 6) {\Huge $\times$};
				\node at (8.5, 7) {\Huge $\times$};
			\end{tikzpicture}
		}
		 &  &
		\scalebox{.6}{
			\begin{tikzpicture}[scale=1]
				\draw (0,0) node { };
				\draw (0,4) node {\LARGE$\longrightarrow$};
			\end{tikzpicture}
		}
		 &  &
		\scalebox{.6}{
			\begin{tikzpicture}[scale=1]
				\dyckpathmn{0,0}{5}{8}{1,1,1,1,0,1,1,0,0,1,1,0,0};
				\draw (.5,.5) node {\LARGE$2$};
				\draw (.5,1.5) node {\LARGE$4$};
				\draw (.5,2.5) node {\LARGE$7$};
				\draw (.5,3.5) node {\LARGE$5$};
				\draw (1.5,4.5) node {\LARGE$1$};
				\draw (1.5,5.5) node {\LARGE$3$};
				\draw (3.5,6.5) node {\LARGE$8$};
				\draw (3.5,7.5) node {\LARGE$6$};
				\dyckpathmnnogrid{0,0}{5}{8}{0,0,1,0,1,1,1,0,1,0,1,1,1};
			\end{tikzpicture}
		}
	\end{align*}
	\caption{A pictorial description of $\iota$ for a $\varnothing$-parking function}
	\label{fig:iota}
\end{figure}

As a corollary, we get the desired $e$-expansion.

\begin{theorem}
	\[ \left. \Delta_{m_\gamma} \Xi e_\lambda \right\rvert_{t=1} = \sum_{p \in \PF^{\gamma}_{\lambda} } q^{\area(p)} e_{\eta(p)}. \]
\end{theorem}

\begin{example}
	Looking at Figure~\ref{fig:(2)-parking-functions}, we have the $e$-expansion
	\begin{align*}
		\left. \Delta_{m_2} \Xi e_{11} \right\rvert_{t=1} & = e_{11} + e_2 + q e_2 + e_{11} + e_{11} + qe_{11} + qe_{11} + q^2 e_{11} + q^2 e_2 + q^3 e_2 \\
		                                                  & = (q^3+q^2+q+1) e_2 + (q^2+2q+3) e_{11}
	\end{align*}
	which is indeed the correct one.
\end{example}

\section{Using different labelings}

Our construction generalizes to a broader framework, in which we use a different family of labelings and a different statistic on words. If these new labelings satisfy certain properties, then the construction we just describe still holds and allows us to give a combinatorial expansion of different families of symmetric functions.

To this end, let us analyze the construction of an element $T \in \LC^{\gamma}_{\lambda,\eta}$ and the computation of its statistic. Let $\vec{w}(T) = (w^1,\dots, w^r) \in \WV(\lambda, \beta)$ with $\beta \in R(\eta)$, and let $w(T) = w^1 \cdots w^r$ be the concatenation, that is, the word in the base rows of $T$ read from left to right. The key property of $\revmaj$ we use in Proposition~\ref{weightpreserving} and Lemma~\ref{splitlemma} is that it is computed by taking a subset of letters of $w(T)$ and summing the number of letters to the right of, and in the same base as, each letter in the subset. But these two results do not depend on the subset itself. In the end, we can give the following definition.

\begin{definition}
	Let $W$ be any set of words, and let $\rho \colon W \rightarrow \mathbb{N}$ be any statistic on $W$. We say that $\rho$ \emph{looks right} if \[ \rho(w) = \sum_{i \in S(w)} (\ell(w)-i) \] for any subset-picking function $S \colon W \rightarrow 2^{\mathbb{N}}$ such that $S(w) \subseteq \{1,\dots, \ell(w)\}$.
\end{definition}

Notice that $\rho$ and $S$ determine each other, meaning that for any $\rho$ that looks right there is a unique subset-picking function $S$ satisfying the identity, and any subset-picking function $S$ defines a statistic $\rho$ that looks right.

We can extend the definition to vectors as follows.

\begin{definition}
	Let $WV$ be any set of word vectors, and let $\vec{\rho} \colon WV \rightarrow \mathbb{N}$ be any statistic on $WV$. We say that $\vec{\rho}$ looks right if, for $\vec{w} = (w^1, \dots, w^r) \in WV$, we have \[ \vec{\rho}(\vec{w}) = \sum_{i=1}^r \rho(w^i) \] where $\rho$ is a statistic that is defined on all the entries of word vectors in $WV$ and it looks right. With an abuse of notation, we write $\rho$ for $\vec{\rho}$.
\end{definition}

Let $W \subseteq \mathbb{N}^n$, let $\LC_{n,\eta}^\gamma = \bigcup_{\lambda \vDash_w n} \LC_{\lambda,\eta}^\gamma$ (i.e.\ there is no restriction on the word), and let \[ \LC_{W,\eta}^\gamma \coloneqq \left\{ T \in \LC_{n,\eta}^\gamma \mid w(T) \in W \right\}. \] Let $\rho$ be any statistic defined on the set of subwords of words in $W$ that looks right. For $T \in \LC_{W, \eta}^\gamma$, with the same notation as in Definition~\ref{def:weight} let \[ \weight(T_i) = \lvert C^i \rvert + \rho(w^i) + u(l^i), \qquad \weight(T) = \sum_{i=1}^{\ell(\eta)} \weight(T_i). \]

For any composition $\beta $, let $\WV(\beta) = \bigcup_{\lambda \vDash_w n} \WV(\lambda, \beta)$, that is, the set of word vectors $\vec{w} = (w^1, \dots, w^{\ell(\beta)})$ such that $\ell(w^i) = \beta_i$. We define \[ W(\beta) = \{\vec{w} \in \WV(\beta) \mid  w^1\cdots w^{\ell(\beta)} \in W \}, \] which is the subset of such vectors that concatenate to a word in $W$. Note that $\rho$ extends to $W(\beta)$ so that it looks right.

Let \[ D_{W, \eta}^{\gamma}(q) = \sum_{\beta \vDash n} \sum_{\vec{w} \in W(\beta)} \sum_{\vec{\nu} \in \PR(\eta,\beta)} q^{\rho(\vec{w})} m_\gamma\left[ \sum_i [\beta_i]_q \right] (-1)^{n-\ell(\beta)} (1-q^{\beta_1}) f_{\vec{\nu}}\left[ \frac{1}{1-q} \right]. \]

We have the following analogue of Theorem~\ref{fixedpointtheorem}.
\begin{theorem}
	Let
	\begin{align*}
		U_{W,\eta}^\gamma = \Big\{ T\in \LC_{W,\eta}^\gamma \mid T & \text{ has no vertical bars, and for all $i$}                               \\
		                                                           & c_1(T_{i+1}) < c_\ell (T_i) + s(T_i;T_{i+1}) + \lvert l(T_i) \rvert \Big\},
	\end{align*}
	where $s(T_i;T_{i+1}) = \lvert S(w^i\cdot w^{i+1}_1) \rvert$. Then \[ D^{\gamma}_{W,\eta}(q) = \sum_{T \in U_{W,\eta}^\gamma} q^{\weight(T)}. \]
\end{theorem}

\begin{proof}
	The same argument as in Proposition~\ref{weightpreserving} and Lemma~\ref{splitlemma} holds if we replace $\asc$ with any statistic that looks right. In this case, we replaced $\asc$ with $\rho$, which looks right by hypothesis, so the statement holds.
\end{proof}

Definitions~\ref{LPPDefinition} and \ref{def:ascent-polyominoes} generalize as follows.

\begin{definition}
	Let $W$ be a set of words of length $n$, let $\rho$ be a statistic that looks right on the subwords of words in $W$, and let $S$ be the corresponding subset-picking function. An $S$-labeled parallelogram polyomino is a triple $(P,Q,w)$ where $w \in W$, and $(P,Q)$ is a parallelogram polyomino such that if $Q$ has no East steps on the line $y = i-1$ (or has only $1$ East step if $i=1$, since the first step of $Q$ must be East), then $i \not \in S(w)$.
\end{definition}

\begin{definition}
	Let $(P, Q, w)$ be an $S$-labeled parallelogram polyomino. Let $\eta \vdash n$ be the partition whose block sizes are the lengths of the maximal streaks of North steps in $P$, in some order. Let \[ \beta_i \coloneqq \# \{ \text{East steps of $Q$ on the line $y = i-1$} \} - \chi(i=1) - \chi(i \in S(w)), \] and let $\gamma$ be the partition obtained by rearranging $(\beta_1, \dots, \beta_n)$ and removing zeros.

	We define $\type(P, Q, w) = (W, \eta, \gamma)$, and call $\PP_{W, \eta}^\gamma$ the set of S-labeled polyominoes of type $(W, \eta, \gamma)$. Once again, the width depends on the cardinality of $S(w)$ and it's not constant through the set.
\end{definition}

Finally, using the same argument as in Theorem~\ref{bijectiontheorem}, we have the following.
\begin{theorem}
	\[ D_{W,\eta}^{\gamma}(q) = \sum_{P \in \PP_{W,\eta}^{\gamma} } q^{\area(P)}. \]
\end{theorem}

In particular, for $S = \Des$, our labeled polyominoes $(P,Q,w)$ enumerated by $D_{\lambda,\eta}^{\gamma}$ can be interpreted in two different ways: in one, the path $Q$ has forced East steps at the descents of $w$, in the other it has forced East steps at the ascents of $w$. Both are valid combinatorial interpretations.

Finally, the map $\iota$ from Definition~\ref{def:iota} also generalizes, by replacing the condition $i \not \in \Asc(w)$ with $i \not \in S(w)$. Of course, the set of $\gamma$-Dyck paths we get in the end will not necessarily have strictly increasing columns, but rather the condition that if $w_i, w_{i+1}$ are in the same column, then $i \in S(w)$.

\section{Applications to Schur functions}

We can now use the idea of choosing different labelings to get similar results for other instances of $\widetilde{\Delta}_{m_\gamma} \Xi F$. One interesting case that our technique can handle is when $F$ is a Schur function (rather than an elementary symmetric function). We now find an $e$-expansion for $\widetilde{\Delta}_{m_\gamma} \Xi s_{\lambda}$.

The idea is to mimic what we did in Section~\ref{sec:preliminary-manipulations}. We have the expansion \[ s_{\lambda}^\ast = \sum_{\mu \vdash n} \frac{\Ht_\mu}{w_\mu} \langle s_{\lambda}^\ast , \Ht_\mu \rangle_\ast = \sum_{\mu \vdash n} \frac{\Ht_\mu}{w_\mu} \langle s_{\lambda'} , \Ht_\mu \rangle, \] where again in the last step we went from the $\ast$-scalar product to the ordinary Hall scalar product.

We can handle the term $\langle s_{\lambda'} , \Ht_\mu \rangle$ by using the specialization at $t=1$ of the Macdonald polynomials \cite{Macdonald-Book-1995, Haglund-Book-2008, Haglund-Haiman-Loehr}. Combining \cite[Equation~(2.26), (2.30)]{Haglund-Book-2008} and specializing $t=1$, we get the formula \[ \left. \langle s_{\lambda} , \Ht_\mu \rangle \right\rvert_{t=1} = \tilde{K}_{\lambda, \mu}(q,1) = \sum_{T \in \SYT(\lambda)} q^{\comaj(T, \mu)}, \] where $\SYT(\lambda)$ is the set of standard Young tableaux of shape $\lambda$, and $\comaj(T, \mu)$ is computed as follows: first, partition the entries of $T$ in blocks according to $\mu$, i.e.\ one block containing the entries going from $1$ to $\mu_1$, one block containing the entries going from $\mu_1 + 1$ to $\mu_1 + \mu_2$, and so on; then, for $1 \leq j \leq \lvert \lambda \rvert$, check if $j$ and $j+1$ are in the same block, say the $i^{\ith}$ block, and check if $j+1$ occurs in a row above $j$ in $T$; if they do, add $\mu_1 + \dots + \mu_i - j$ to $\comaj(T, \mu)$.

We want to convert this statement into one about words. Recall that a lattice word $w = (w_1,\dots, w_n)$ is a word such that for every $i$ and $j$, \[ m_{j+1}(w_1,\dots,w_i) \leq m_j(w_{1},\dots, w_{i}). \] This means that when read from the left, the number of $j$'s encountered is, at every moment, greater than or equal to the number of $j+1$'s. If $m_j (w) = \lambda_j$, then we say $w$ has content $\lambda$, and write $w \in \LW(\lambda)$. Notice that, since $w$ is a lattice word, then $\lambda$ must be a partition. We have the following.

\begin{lemma}
	Lattice words encode standard Young tableaux.
\end{lemma}

\begin{proof}
	Given a lattice word $w \in \LW(\lambda)$, let $T$ be the standard Young tableau obtained by putting $i$ in the $w_i^{\ith}$ row. Since $w$ has content $\lambda$, then $T$ has shape $\lambda$, and since $w$ is a lattice word, then $T$ is indeed a standard Young tableau. It is clear that this construction is reversible and hence bijective. See Figure~\ref{fig:syt} for an example.
\end{proof}

\begin{figure}[!ht]
	\centering
	\begin{tikzpicture}[scale=1]
		\partitionfr{4,3,2}
		\draw (0.5, 0.5) node {$1$};
		\draw (1.5, 0.5) node {$2$};
		\draw (2.5, 0.5) node {$4$};
		\draw (3.5, 0.5) node {$7$};
		\draw (0.5, 1.5) node {$3$};
		\draw (1.5, 1.5) node {$6$};
		\draw (2.5, 1.5) node {$9$};
		\draw (0.5, 2.5) node {$5$};
		\draw (1.5, 2.5) node {$8$};
	\end{tikzpicture}
	\caption{The standard Young tableau of shape $(4,3,2)$ encoded by the lattice word $112132132$}
	\label{fig:syt}
\end{figure}

For any composition $\beta$, recall that $\WV(\beta)$ is the set of word vectors $\vec{w} = (w^1, \dots, w^{\ell(\beta)})$ such that $\ell(w^i) = \beta_i$. We define \[ \LW(\lambda, \beta) = \{\vec{w} \in \WV(\beta) \mid  w^1\cdots w^{\ell(\beta)} \in \LW(\lambda) \}, \] which is the subset of vectors that concatenate to a lattice word with content $\lambda$. Recall that \[ \revmaj(\vec{w}) = \sum_{i=1}^{\ell(\beta)} \revmaj(w^i). \] We have the following.

\begin{lemma}
	\label{lem:tableau-to-lattice-stat}
	If $w \in \LW(\lambda)$ encodes $T \in \SYT(\lambda)$, for $\vec{w} \in \LW(\lambda, \beta)$ that concatenates to $w$, we have \[ \revmaj(\vec{w}) = \comaj(T, \beta). \]
\end{lemma}

\begin{proof}
	Let $1 \leq i \leq \ell(\beta)$, let $\Sigma_i(\beta) \coloneqq \beta_1 + \dots + \beta_i$, and let \[ \Asc_{\beta,i}(w) \coloneqq \{ \Sigma_{i-1}(\beta) + r \mid r \in \Asc(w^i) \}, \] that is, the set $\Asc(w^i)$ relabeled so that its entries give the positions of the ascents in $w$ rather than in $w^i$. By definition, \[ \revmaj(w^i) = \sum_{r \in \Asc(w^i)} (\beta_i - r) = \sum_{j \in \Asc_{\beta,i}(w)} (\Sigma_i(\beta) - j). \] If $j \in \Asc_{\beta,i}(w)$, then necessarily $j+1 \leq \Sigma_i(\beta)$ (because $w_j$ must occur within $w^i$), so $j$ and $j+1$ belong to the same block of $T$ according to $\beta$; moreover, we have $w_j < w_{j+1}$, so $j+1$ occurs in a row above $j$ in $T$. This means that the contribution of $j$ to $\comaj(T, \beta)$ is exactly $\Sigma_{i}(\beta) - j$, so taking the sum over all $j$ we get \[ \revmaj(\vec{w}) = \comaj(T, \beta) \] as desired.
\end{proof}

\begin{figure}[!ht]
	\centering
	\includegraphics{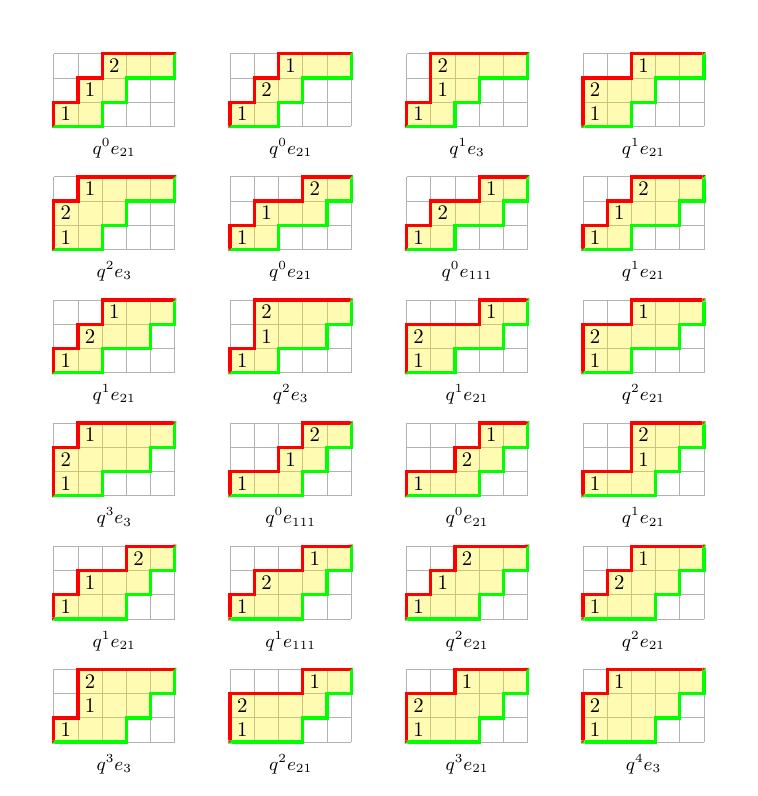}
	\caption{The $e$-expansion for $\widetilde{\Delta}_{m_1} \Xi s_{21}$}
	\label{FigureFullExample}
\end{figure}

It is important to remark that \[ \Ht_\mu[X;q,1] = (q;q)_\mu h_\mu \left[ \frac{1}{1-q} \right], \] which is a multiplicative basis. This means that, when evaluating $t=1$, the order of the parts of $\mu$ doesn't matter and we can use compositions rather than partitions, the combinatorial argument being exactly the same. Putting everything together, we have the following expansion.

\begin{proposition}
	\[ \left\langle s_{\lambda} , (q;q)_\beta h_\beta \left[ \frac{1}{1-q} \right] \right\rangle = \sum_{\vec{w} \in \LW(\lambda, \beta)} \revmaj(\vec{w}). \]
\end{proposition}

Using this expansion, the same argument we used in Section~\ref{sec:preliminary-manipulations} to prove Proposition~\ref{firstexpansion} leads us to the following.

\begin{proposition}
	For any $\lambda \vdash n$ and $\gamma \vdash m$, we have
	\[ \widetilde{\Delta}_{m_\gamma} \Xi s_{\lambda} = \sum_{\eta \vdash n } e_\eta ~C^{\gamma}_{s_\lambda,\eta}(q), \]
	where
	\begin{align*}
		C^{\gamma}_{s_\lambda, \eta} = \sum_{\beta \vDash n} \sum_{\vec{w} \in \LW(\lambda', \beta)} \sum_{\vec{\nu} \in \PR(\eta,\beta)} & q^{\revmaj(\vec{w})} m_\gamma\left[ \sum_i [\beta_i]_q \right]                         \\
		                                                                                                                                  & \times (-1)^{n-\ell(\beta)} (1-q^{\beta_1}) f_{\vec{\nu}}\left[ \frac{1}{1-q} \right].
	\end{align*}
\end{proposition}

Since our statistic is still $\revmaj$, we still have $S(w) = \Asc(w)$, so we can repeat the whole construction starting from lattice words instead, and get lattice $\gamma$-Dyck paths. In the end, we get the following.

\begin{theorem}
	\[ \widetilde{\Delta}_{m_\gamma} \Xi s_{\lambda} = \sum_{p \in \LPF^{\gamma}_{\lambda'} } q^{\area(p)} e_{\eta(p)}. \]
\end{theorem}

\begin{example}
	For $\lambda = (2,1)$ and $\gamma = (1)$, we have \[ \widetilde{\Delta}_{m_\gamma} \Xi s_{\lambda} = (q+2) e_{111} + (q^3+4q^2+6q+4) e_{21} + (q^4+2q^3+2q^2+q) e_3. \] The corresponding combinatorial expansion is shown in Figure~\ref{FigureFullExample}, and we indeed see that they coincide.
\end{example}

\section{Special cases}

Our construction yields a lot of special cases that are of interest and that find matches in the literature. In this section, we aim to go through some of them.

\subsection{The extended Delta theorem when \texorpdfstring{$t=1$}{t=1}}

From the fundamental fact that $\Xi e_n = e_n$, as a corollary of our construction we obtain a special case of the results in \cite{Hicks-Romero-2018} and \cite{Romero-Thesis}. Here, the polyominoes in $\PP^{\gamma}_{(n),\eta}$ are precisely the set of parallelogram polyominoes whose top path has vertical segments whose lengths rearrange to $\eta$, and whose bottom path has horizontal segments (ignoring the first step of the path) whose lengths rearrange to $\gamma$. There is only one word labeling the polyominoes, consisting of only $1$'s. There are no descents or ascents in this word. Our bijection to $\gamma$-parking functions produces a pair of paths $(P',Q')$ where $P$ has no consecutive North steps, and the bottom path has lengths rearranging to $\gamma + 1^n$.

The same result can be used to derive the $t=1$ case of the Extended Delta Theorem (\cite{Haglund-Remmel-Wilson-2018, DAdderio-Mellit-Compositional-Delta-2020, Blasiak-Haiman-Morse-Pun-Seeling-Extended-Delta-2021}). In fact, there is an explicit bijection between appropriate families of our objects and the objects appearing in the theorem; it states the following.

\begin{theorem}[Extended Delta Theorem]
	For $m, n, k \in \mathbb{N}$, we have the monomial expansion \[ \Delta_{h_m e_{n-k}} e_n = \sum_{p \in \LD(n,m)^{\ast k}} q^{\dinv(p)} t^{\area(p)} x^p. \]
\end{theorem}

Here, $\LD(n,m)^{\ast k}$ is the set of labeled Dyck paths of size $m+n$, with $m$ zero labels (that cannot be in the first column) and $k$ decorated double rises (i.e\ North steps preceded by North steps; the first step also counts as a double rise); $\dinv(p)$ is the number of \emph{diagonal inversions} of the path (that we do not define); $\area(p)$ is the number of whole squares between the path and the main diagonal in rows that do not contain decorations; and $x^p$ is the product of the variables indexed by labels of the path, where $x_0 = 1$. For our purposes, it will be more convenient to consider paths with decorated double falls (i.e\ East steps followed by East steps; the last step also counts as a double fall), which is clearly equivalent.

Note that the symmetric function is also symmetric in $q$ and $t$, so we can set $q=1, t=q$ and since the area does not depend on the labels, this gives us an $e$-expansion of the symmetric function when $t=1$. We are going to show that this expansion coincides with ours. Since the content of the labeling is $(n)$, it is more convenient to use ascent polyominoes, rather than $\gamma$-parking functions, and disregard the labeling altogether.

Indeed, $h_m e_{n-k}$ is monomial-positive: the coefficient of $m_\gamma$ is the number of ways one can fill $\ell(\gamma)$ ordered boxes with $m$ green balls and $n-k$ blue balls, such that the $i^{\ith}$ box contains $\gamma_i$ balls in total and no box contains $2$ or more green balls. In our language, this means that the $e$-expansion of $\Delta_{h_m e_{n-k}} e_n$ will be given by all the $\gamma$-parking functions of content $(n)$ such that the bottom path is obtained by starting with $EN^n$, then inserting $n-k$ East steps in different rows colored green, and then inserting $m$ more blue East steps in any possible way.

The bijection we give is essentially a combination of \cite[Theorem~6.17]{DAdderio-Iraci-VandenWyngaerd-TheBible-2019}
(see also \cite[Theorem~6.1]{DAdderio-Iraci-VandenWyngaerd-GenDeltaSchroeder-2019}) and \cite[Theorem~2]{Romero-Thesis}, disregarding the labels. We will not give a rigorous proof, but we will go through an example in detail, and the generalization follows.

\begin{example} \label{ExampleExtendedDelta}
	Let $n=7$, $k=4$, and $m=3$.

	\textbf{Step 1.} We start with the bottom path $NE^7$, we add $n-k=3$ East steps in different rows (in green in the picture), and then we add $m=3$ East steps anywhere (in blue in the picture). To compute the area, we ignore the cells that touch the bottom path; the cells that do contribute are shaded.

	\begin{center}
		\scalebox{0.6}{
			\begin{tikzpicture}
				\draw[draw=none, use as bounding box] (-1, -1) rectangle (8,8);
				\draw[gray!60, thin] (0,0) grid (7,7);

				\filldraw[yellow, opacity=0.3] (0,0) -- (0,1) -- (1,1) -- (1,3) -- (4,3) -- (4,5) -- (5,5) -- (5,6) -- (6,6) -- (6,7) -- (5,7) -- (4,7) -- (4,6) -- (3,6) -- (2,6) -- (2,5) -- (2,4) -- (1,4) -- (0,4) -- (0,3) -- (0,2) -- (0,1) -- (0,0);

				\draw[black, sharp <-sharp >, sharp angle = -45, line width=1.6pt] (0,0) -- (0,1) -- (0,2) -- (0,3) -- (0,4) -- (1,4) -- (2,4) -- (2,5) -- (2,6) -- (3,6) -- (4,6) -- (4,7) -- (5,7) -- (6,7) -- (7,7);

				\draw[black, sharp <-sharp >, sharp angle = 45, line width=1.6pt] (0,0) -- (1,0) -- (2,0) -- (2,1) -- (2,2) -- (3,2) -- (4,2) -- (5,2) -- (5,3) -- (5,4) -- (6,4) -- (6,5) -- (7,5) -- (7,6) -- (7,7);

				\draw[green, sharp <-sharp >, sharp < angle = 0, sharp > angle = 45, line width=1.6pt] (1,0) -- (2,0);
				\draw[green, sharp <-sharp >, sharp < angle = 45, sharp > angle = 0, line width=1.6pt] (2,2) -- (3,2);
				\draw[green, sharp <-sharp >, sharp < angle = 45, sharp > angle = 45, line width=1.6pt] (5,4) -- (6,4);

				\draw[blue, sharp <-sharp >, sharp < angle = 0, sharp > angle = 45, line width=1.6pt] (3,2) -- (5,2);
				\draw[blue, sharp <-sharp >, sharp < angle = 45, sharp > angle = 45, line width=1.6pt] (6,5) -- (7,5);
			\end{tikzpicture}
		}
	\end{center}

	\textbf{Step 2.} For each blue step (the ones coming from $h_m$) we insert a North step in the bottom path right before it, and insert a North step in the top path in the column one unit to the left. We mark the North steps added to the top path with a $\bullet$ symbol, to denote that they will be the ``empty'' valleys in the final picture. We highlight these steps in red. Notice that the area does not change in the process, as the number of shaded cells in each column does not change.

	\begin{center}
		\scalebox{0.6}{
			\begin{tikzpicture}
				\draw[draw=none, use as bounding box] (-1, -1) rectangle (8,11);
				\draw[gray!60, thin] (0,0) grid (7,10);

				\filldraw[yellow, opacity=0.3] (0,1) -- (0,4) -- (2,4) -- (2,7) -- (3,7) -- (3,8) -- (4,8) -- (4,9) -- (5,9) -- (5,10) -- (6,10) -- (6,9) -- (5,9) -- (5,7) -- (4,7) -- (4,5) -- (3,5) -- (3,4) -- (2,4) -- (2,3) -- (1,3) -- (1,1) -- (0,1);

				\draw[black, sharp <-sharp >, sharp < angle = 45, sharp > angle = 45, line width=1.6pt] (0,0) -- (0,4) -- (2,4);
				\draw[red, sharp <-sharp >, sharp < angle = -45, sharp > angle = 0, line width=1.6pt] (2,4) -- (2,5);
				\draw[black, sharp <-sharp >, sharp < angle = 0, sharp > angle = 45, line width=1.6pt] (2,5) -- (2,7) -- (3,7);
				\draw[red, sharp <-sharp >, sharp < angle = -45, sharp > angle = -45, line width=1.6pt] (3,7) -- (3,8);
				\draw[black, sharp <-sharp >, sharp < angle = 45, sharp > angle = 45, line width=1.6pt] (3,8) -- (4,8) -- (4,9) -- (5,9);
				\draw[red, sharp <-sharp >, sharp < angle = -45, sharp > angle = -45, line width=1.6pt] (5,9) -- (5,10);
				\draw[black, sharp <-sharp >, sharp < angle = 45, sharp > angle = -45, line width=1.6pt] (5,10) -- (7,10);

				\draw[black, sharp <-sharp >, sharp < angle = -45, sharp > angle = 0, line width=1.6pt] (0,0) -- (1,0);
				\draw[green, sharp <-sharp >, sharp < angle = 0, sharp > angle = 45, line width=1.6pt] (1,0) -- (2,0);
				\draw[black, sharp <-sharp >, sharp < angle = -45, sharp > angle = -45, line width=1.6pt] (2,0) -- (2,2);
				\draw[green, sharp <-sharp >, sharp angle = -45, sharp > angle = 45, line width=1.6pt] (2,2) -- (3,2);
				\draw[red, sharp <-sharp >, sharp < angle = -45, sharp > angle = -45, line width=1.6pt] (3,2) -- (3,3);
				\draw[blue, sharp <-sharp >, sharp < angle = 45, sharp > angle = 45, line width=1.6pt] (3,3) -- (4,3);
				\draw[red, sharp <-sharp >, sharp < angle = -45, sharp > angle = -45, line width=1.6pt] (4,3) -- (4,4);
				\draw[blue, sharp <-sharp >, sharp < angle = 45, sharp > angle = 45, line width=1.6pt] (4,4) -- (5,4);
				\draw[black, sharp <-sharp >, sharp < angle = -45, sharp > angle = -45, line width=1.6pt] (5,4) -- (5,5) -- (5,6);
				\draw[green, sharp <-sharp >, sharp < angle = 45, sharp > angle = 45, line width=1.6pt] (5,6) -- (6,6);
				\draw[black, sharp <-sharp >, sharp < angle = -45, sharp > angle = 0, line width=1.6pt] (6,6) -- (6,7);
				\draw[red, sharp <-sharp >, sharp < angle = 0, sharp > angle = -45, line width=1.6pt] (6,7) -- (6,8);
				\draw[blue, sharp <-sharp >, sharp < angle = 45, sharp > angle = 45, line width=1.6pt] (6,8) -- (7,8);
				\draw[black, sharp <-sharp >, sharp < angle = -45, sharp > angle = 45, line width=1.6pt] (7,8) -- (7,9) -- (7,10);

				\node at (2.5, 4.5) {\Large $\bullet$};
				\node at (3.5, 7.5) {\Large $\bullet$};
				\node at (5.5, 9.5) {\Large $\bullet$};
			\end{tikzpicture}
		}
	\end{center}

	\textbf{Step 3.} For each pair of consecutive North steps in the bottom path, we insert an East step in between them, and an East step in the top path in the same column, on the right side of the already present East step. We mark these extra East steps with a $\ast$ symbol. We do the same in the first row if there is no green step there. By construction, in the top path, these added steps have another East step to the left. We do not shade the area in the newly introduced column. We won't need the previous colours anymore, so we highlight again in red these new steps.

	\begin{center}
		\scalebox{0.6}{
			\begin{tikzpicture}
				\draw[draw=none, use as bounding box] (-1, -1) rectangle (12,11);
				\draw[gray!60, thin] (0,0) grid (11,10);

				\filldraw[yellow, opacity=0.3] (0,1) -- (0,4) -- (2,4) -- (2,3) -- (1,3) -- (1,1) -- (0,1) (3,4) -- (3,7) -- (4,7) -- (4,8) -- (5,8) -- (5,9) -- (6,9) -- (6,7) -- (5,7) -- (5,5) -- (4,5) -- (4,4) -- (3,4) (7,9) -- (7,10) -- (8,10) -- (8,9);

				\draw[black, sharp <-sharp >, sharp < angle = 45, sharp > angle = -0, line width=1.6pt] (0,0) -- (0,4) -- (2,4);
				\draw[red, sharp <-sharp >, sharp < angle = 0, sharp > angle = 45, line width=1.6pt] (2,4) -- (3,4);
				\draw[black, sharp <-sharp >, sharp < angle = -45, sharp > angle = -0, line width=1.6pt] (3,4) -- (3,5) -- (3,7) -- (4,7) -- (4,8) -- (5,8) -- (5,9) -- (6,9);
				\draw[red, sharp <-sharp >, sharp < angle = 0, sharp > angle = 45, line width=1.6pt] (6,9) -- (7,9);
				\draw[black, sharp <-sharp >, sharp < angle = -45, sharp > angle = -0, line width=1.6pt] (7,9)-- (7,10) -- (8,10);
				\draw[red, sharp <-sharp >, sharp < angle = 0, sharp > angle = -0, line width=1.6pt] (8,10) -- (9,10);
				\draw[black, sharp <-sharp >, sharp < angle = 0, sharp > angle = -0, line width=1.6pt] (9,10) -- (10,10);
				\draw[red, sharp <-sharp >, sharp < angle = 0, sharp > angle = -45, line width=1.6pt] (10,10) -- (11,10);

				\draw[black, sharp <-sharp >, sharp < angle = -45, sharp > angle = -45, line width=1.6pt] (0,0) -- (1,0) -- (2,0) -- (2,1);
				\draw[red, sharp <-sharp >, sharp < angle = 45, sharp > angle = 45, line width=1.6pt] (2,1) -- (3,1);
				\draw[black, sharp <-sharp >, sharp < angle = -45, sharp > angle = -45, line width=1.6pt]  (3,1) -- (3,2) -- (4,2) -- (4,3) -- (5,3) -- (5,4) -- (6,4) -- (6,5);
				\draw[red, sharp <-sharp >, sharp < angle = 45, sharp > angle = 45, line width=1.6pt] (6,5) -- (7,5);
				\draw[black, sharp <-sharp >, sharp < angle = -45, sharp > angle = -45, line width=1.6pt] (7,5) -- (7,6) -- (8,6) -- (8,7);
				\draw[red, sharp <-sharp >, sharp < angle = 45, sharp > angle = 45, line width=1.6pt] (8,7) -- (9,7);
				\draw[black, sharp <-sharp >, sharp < angle = -45, sharp > angle = -45, line width=1.6pt] (9,7) -- (9,8) -- (10,8) -- (10,9);
				\draw[red, sharp <-sharp >, sharp < angle = 45, sharp > angle = 45, line width=1.6pt] (10,9) -- (11,9);
				\draw[black, sharp <-sharp >, sharp < angle = -45, sharp > angle = 45, line width=1.6pt] (11,9) -- (11,10);

				\node at (3.5, 4.5) {\Large $\bullet$};
				\node at (4.5, 7.5) {\Large $\bullet$};
				\node at (7.5, 9.5) {\Large $\bullet$};

				\node at (2.5, 4.5) {\Large $\ast$};
				\node at (6.5, 9.5) {\Large $\ast$};
				\node at (8.5, 10.5) {\Large $\ast$};
				\node at (10.5, 10.5) {\Large $\ast$};
			\end{tikzpicture}
		}
	\end{center}

	\textbf{Step 4.} Now, we move the $\ast$ symbols one column to the left, and we move all the shaded cells that end up under a $\ast$ one column to the right. Then, we draw the diagonal $y=x$.

	\begin{center}
		\scalebox{0.6}{
			\begin{tikzpicture}
				\draw[draw=none, use as bounding box] (-1, -1) rectangle (12,11);
				\draw[gray!60, thin] (0,0) grid (11,10) (0,0) -- (10,10);

				\filldraw[yellow, opacity=0.3] (0,1) -- (0,4) -- (1,4) -- (1,1)  (2,3) -- (2,4) -- (3,4) -- (3,3) (3,4) -- (3,7) -- (4,7) -- (4,8) -- (5,8) -- (5,5) -- (4,5) -- (4,4) -- (3,4) (6,7) -- (6,9) -- (7,9) -- (7,7) (9,9) -- (9,10) -- (8,10) -- (8,9);

				\draw[black, sharp <-sharp >, sharp angle = -45, line width=1.6pt] (0,0) -- (0,4) -- (2,4) -- (3,4) -- (3,5) -- (3,7) -- (4,7) -- (4,8) -- (5,8) -- (5,9) -- (6,9) -- (7,9) -- (7,10) -- (8,10) -- (9,10)-- (10,10)  -- (11,10);

				\draw[black, sharp <-sharp >, sharp angle = 45, line width=1.6pt] (0,0) -- (1,0) -- (2,0) -- (2,1) -- (3,1) -- (3,2) -- (4,2) -- (4,3) -- (5,3) -- (5,4) -- (6,4) -- (6,5) -- (7,5) -- (7,6) -- (8,6) -- (8,7) -- (9,7) -- (9,8) -- (10,8) -- (10,9) -- (11,9) -- (11,10);

				\node at (3.5, 4.5) {\Large $\bullet$};
				\node at (4.5, 7.5) {\Large $\bullet$};
				\node at (7.5, 9.5) {\Large $\bullet$};

				\node at (1.5, 4.5) {\Large $\ast$};
				\node at (5.5, 9.5) {\Large $\ast$};
				\node at (7.5, 10.5) {\Large $\ast$};
				\node at (9.5, 10.5) {\Large $\ast$};
			\end{tikzpicture}
		}
	\end{center}

	\textbf{Step 5.} The part of the top path going from $(0,0)$ to $(m+n,m+n)$ is now a Dyck path. By construction, it has $m$ marked valleys and $n-k$ decorated double falls (i.e.\ pairs of consecutive East steps). The area of the path is the number of whole cells between the path and the diagonal that are not in columns containing a $\ast$, which are exactly the shaded cells. The lengths of the vertical segments, ignoring the parts containing a marked valley, are also preserved. It follows that this procedure yields a bijection that preserves both the $\area$ and the $e$-composition, as desired.
\end{example}

\subsection{The parking function case}

When $\lambda = 1^n$ and $\gamma = \varnothing$, then elements in $\PF^\varnothing_{1^n}$ are constructed by selecting a $\varnothing$-Dyck path and writing the numbers $1, \dots, n$ along the North segments of the top path so that the columns are increasing. This is precisely the set of parking functions in the classical sense.

\begin{figure}[!ht]
	\centering
	\scalebox{.6}{
		\begin{tikzpicture}[scale=1]
			\draw[gray!60, thin] (0,0) grid (9,8) (0,0) -- (8,8);
			\filldraw[yellow, opacity=0.3] (0,0) -- (1,0) -- (2,0) -- (2,1) -- (3,1) -- (3,2) -- (4,2) -- (4,3) -- (5,3) -- (5,4) -- (6,4) -- (6,5) -- (7,5) -- (7,6) -- (8,6) -- (8,7) -- (9,7) -- (9,8) -- (8,8) -- (7,8) -- (6,8) -- (6,7) -- (5,7) -- (5,6) -- (4,6) -- (3,6) -- (3,5) -- (3,4) -- (2,4) -- (1,4) -- (1,3) -- (0,3) -- (0,2) -- (0,1) -- (0,0);
			\draw[red, sharp <-sharp >, sharp angle = -45, line width=1.6pt] (0,0) -- (0,1) -- (0,2) -- (0,3) -- (1,3) -- (1,4) -- (2,4) -- (3,4) -- (3,5) -- (3,6) -- (4,6) -- (5,6) -- (5,7) -- (6,7) -- (6,8) -- (7,8) -- (8,8) -- (9,8);
			\draw[green, sharp <-sharp >, sharp angle = 45, line width=1.6pt] (0,0) -- (1,0) -- (2,0) -- (2,1) -- (3,1) -- (3,2) -- (4,2) -- (4,3) -- (5,3) -- (5,4) -- (6,4) -- (6,5) -- (7,5) -- (7,6) -- (8,6) -- (8,7) -- (9,7) -- (9,8);
			\node at (0.5,0.5) {\LARGE$2$};
			\node at (0.5,1.5) {\LARGE$4$};
			\node at (0.5,2.5) {\LARGE$7$};
			\node at (1.5,3.5) {\LARGE$5$};
			\node at (3.5,4.5) {\LARGE$1$};
			\node at (3.5,5.5) {\LARGE$3$};
			\node at (5.5,6.5) {\LARGE$8$};
			\node at (6.5,7.5) {\LARGE$6$};
		\end{tikzpicture}
	}
	\caption{A parking function of size $8$}
	\label{fig:parking-function}
\end{figure}
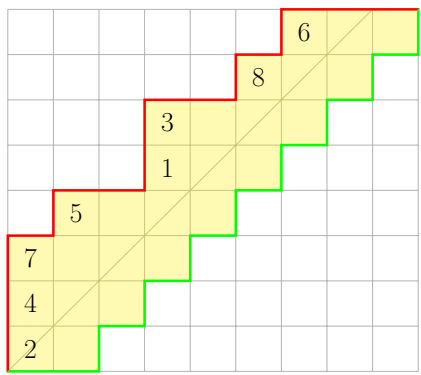

Figure~\ref{fig:parking-function} shows an element of $\PF_8 = \PF_{(8)}^\varnothing$ (the parking functions of size $8$), drawn in the style of our $\varnothing$-parking functions. Its area is $10$.

Indeed, we have \[ \Xi e_{1^n} \vert_{t=1} = \sum_{p \in \PF_n} q^{\area(p)} e_{\eta(p)}, \] which suggests that the combinatorial description given in \cite[Conjecture~6.4]{DAdderio-Iraci-LeBorgne-Romero-VandenWyngaerd-2022} should hold. In fact, it is known that spanning trees on the complete graph on $\{0, 1, \dots, n\}$ $K_{n+1}$, rooted at $0$, $q$-weighted with $\kappa$-inversions, are in bijection with parking functions of size $n$, $q$-weighed with the area. To prove the conjecture, one should show that the $e$-expansion (or, likely easier, the $m$-expansion) is preserved by the bijection.

We actually have that \cite[Conjecture~6.1]{DAdderio-Iraci-LeBorgne-Romero-VandenWyngaerd-2022}, which is very close to this statement, is actually solved for the symmetric function $\Theta_{e_{1^n}} e_1$ (evaluated at $t=1$), of which it provides a monomial expansion. At the moment the relation between the $e$-expansion we just provided and the monomial expansion of \cite[Theorem~4.5]{DAdderio-Iraci-LeBorgne-Romero-VandenWyngaerd-2022} remains unclear.

\subsection{\texorpdfstring{$G$}{G}-parking functions, labeled Dyck paths, tiered trees}

It might be possible to apply the construction of the previous subsection in a much broader context. Indeed, we have a statistc-preserving map between $\varnothing$-Dyck paths with content $\alpha$ and some subset of $\alpha$-tiered trees rooted at $0$ (as in \cite[Conjecture~6.4]{DAdderio-Iraci-LeBorgne-Romero-VandenWyngaerd-2022}), which we can obtain using results about $G$-parking functions \cite{Bak-Chao-Kurt-G-Parking-1987, Dhar-1990} (see also \cite{Perkinson-Yang-Yu-2017}) for appropriate multipartite graphs.

This connection may give a proof of \cite[Conjecture~6.4]{DAdderio-Iraci-LeBorgne-Romero-VandenWyngaerd-2022}, provided that one finds a way to translate the $e$-expansion into the monomial expansion. It is also possible that investigating this further would lead to the discovery of a bistatistic on rooted tiered trees that describes the symmetric function in full (without the need to specialize $t=1$).

\subsection{The two-part case}

When $\lambda = (n,m)$ and $\gamma = \varnothing$, our objects are exactly the so-called \emph{two-car parking functions}, that is, labeled Dyck paths of size $m+n$ with $n$ labels equal to $1$ and $m$ labels equal to $2$. By \cite[Theorem~3.11]{DAdderio-Iraci-VandenWyngaerd-TheBible-2019}, we know that these objects are in fact in bijection with (unlabeled) parallelogram polyominoes of size $(m+1) \times (n+1)$, and the bijection preserves the area.

In light of the statement in the previous subsection, this is not surprising, as we know that parallelogram polyominoes are in bijection with (unlabeled) $(n,1,m)$-tiered rooted trees \cite[Proposition~7.4]{DAdderio-Iraci-LeBorgne-Romero-VandenWyngaerd-2022}. In fact, the bijection is given for labeled objects; as before, this potentially allows us to get a monomial expansion, but this does not seem to coincide with the $e$-expansion provided by our theorems. As for the more general case of the previous subsection, this matter is worth investigating.

It might be worth noting that (unlabeled) parallelogram polyominoes are also in bijection with Dyck words in the alphabet $\{ 0 \leq \overline{0} \leq 1 \leq \overline{1} \leq 2 \leq \dots \}$, starting with either $0$ or $\overline{0}$, with $m$ non-barred letters and $n$ barred letters, with an $\area$ statistic given by the sum of the letters.

\section{Concluding remarks}

We should remark that our identities have multiple consequences. For one, for any $F$ which has a positive expansion in terms of the monomial basis, we have found that $\widetilde{\Delta}_F \Xi e_{\lambda}$ and $\widetilde{\Delta}_{F} \Xi s_\lambda$ are $e$-positive, and their expansion can be given in terms of $\gamma$-parking functions. When $F$ equals $s_\mu$, $e_\mu$, $h_\mu$  or $(-1)^{\lvert \mu \rvert -\ell(\mu)} f_\mu$, we can get a set of combinatorial objects by labeling the bottom path of $\gamma$-parking functions (as is done in \cite{Romero-Thesis} and in Example \ref{ExampleExtendedDelta}). Furthermore, if $G$ expands positively in terms of the Schur basis, then we have also found that $\widetilde{\Delta}_{F} \Xi G$ is $e$-positive. We can similarly get expansions for these symmetric functions in terms of lattice $\gamma$-parking functions.

Another important fact, is that many of these symmetric functions become $e$-positive after the substitution $q=1+u$ (without substituting $t=1$). For instance, $\Delta_{m_\gamma} \Xi e_n$ seems to exhibit this $e$-positivity phenomenon.
\begin{conjecture}
	\label{conjecture-epositivity}

	If $F$ is monomial positive, and $G$ is Schur positive, then \[ \Delta_F \Xi G \Big|_{q \rightarrow 1+u} \] is $e$-positive, i.e.\ the coefficients of the $e$-expansion are polynomials in $\mathbb{N}[u,t]$.

	In particular, if $G$ is $e$-positive, the same result holds.
\end{conjecture}

The case $\gamma = (k)$ and $G = e_n$ follows from the proof of the Delta theorem and the $e$-positivity on vertical strip $LLT$ polynomials proved in \cite{DAdderio-ePositivity-2020}.

If this conjecture is true, then our combinatorial objects would also give the combinatorial expansions for these symmetric functions. To be more precise, these symmetric functions would enumerate the number of $\gamma$-parking functions (or lattice $\gamma$-parking functions) where some subset of its area cells are chosen. It remains to find a $t$ statistic that would give the entire symmetric function.

\begin{conjecture}
	\label{conjecture-epositivity-t}

	There exists a statistic $\mathsf{tstat}$ from pairs $(p,S)$ with $p \in \PF^{\gamma}_{\lambda}$ and $S$ subset of the area cells of $p$, such that \[
		\left. \Delta_{m_\gamma} \Xi e_{\lambda} \right\rvert_{q \rightarrow 1+u} = \sum_{p \in \PF^{\gamma}_{\lambda} } u^{\# S} t^{\mathsf{tstat}(p,S)} e_{\eta(p)}. \]
\end{conjecture}

\section{Acknowledgements}
M. Romero was partially supported by the NSF Mathematical Sciences Postdoctoral Research Fellowship DMS-1902731.

\bibliographystyle{amsalpha}
\bibliography{bibliography}

\begin{bibdiv}
\begin{biblist}

\bib{Bak-Chao-Kurt-G-Parking-1987}{article}{
      author={Bak, Per},
      author={Tang, Chao},
      author={Wiesenfeld, Kurt},
       title={Self-organized criticality: {An} explanation of the 1/f noise},
        date={1987-07},
     journal={Physical Review Letters},
      volume={59},
      number={4},
       pages={381\ndash 384},
         url={https://link.aps.org/doi/10.1103/PhysRevLett.59.381},
}

\bib{Bergeron-Garsia-ScienceFiction-1999}{incollection}{
      author={Bergeron, Fran\c{c}ois},
      author={Garsia, Adriano~M.},
       title={Science fiction and {M}acdonald's polynomials},
        date={1999},
   booktitle={Algebraic methods and {$q$}-special functions (montr\'eal, {QC},
  1996)},
      series={CRM Proc. Lecture Notes},
      volume={22},
   publisher={Amer. Math. Soc., Providence, RI},
       pages={1\ndash 52},
      review={\MR{1726826}},
}

\bib{Bergeron-Garsia-Haiman-Tesler-Positivity-1999}{article}{
      author={Bergeron, Fran\c{c}ois},
      author={Garsia, Adriano~M.},
      author={Haiman, Mark},
      author={Tesler, Glenn},
       title={Identities and positivity conjectures for some remarkable
  operators in the theory of symmetric functions},
        date={1999},
        ISSN={1073-2772},
     journal={Methods Appl. Anal.},
      volume={6},
      number={3},
       pages={363\ndash 420},
         url={https://doi.org/10.4310/MAA.1999.v6.n3.a7},
        note={Dedicated to Richard A. Askey on the occasion of his 65th
  birthday, Part III},
      review={\MR{1803316}},
}

\bib{Blasiak-Haiman-Morse-Pun-Seeling-Extended-Delta-2021}{article}{
      author={Blasiak, Jonah},
      author={Haiman, Mark},
      author={Morse, Jennifer},
      author={Pun, Anna},
      author={Seelinger, George~H.},
       title={A proof of the {E}xtended {D}elta {C}onjecture},
        date={2021-02-17},
     journal={ArXiv e-prints},
      eprint={2102.08815},
}

\bib{Carlsson-Mellit-ShuffleConj-2018}{article}{
      author={Carlsson, Erik},
      author={Mellit, Anton},
       title={A proof of the shuffle conjecture},
        date={2018},
        ISSN={0894-0347},
     journal={J. Amer. Math. Soc.},
      volume={31},
      number={3},
       pages={661\ndash 697},
         url={https://doi.org/10.1090/jams/893},
      review={\MR{3787405}},
}

\bib{DAdderio-ePositivity-2020}{article}{
      author={D'Adderio, Michele},
       title={e-positivity of vertical strip llt polynomials},
        date={2020},
        ISSN={0097-3165},
     journal={Journal of Combinatorial Theory, Series A},
      volume={172},
       pages={105212},
  url={https://www.sciencedirect.com/science/article/pii/S0097316520300042},
}

\bib{DAdderio-Iraci-LeBorgne-Romero-VandenWyngaerd-2022}{article}{
      author={{D'Adderio}, Michele},
      author={{Iraci}, Alessandro},
      author={{Le Borgne}, Yvan},
      author={{Romero}, Marino},
      author={{Vanden Wyngaerd}, Anna},
       title={{Tiered trees and Theta operators}},
        date={2022-02},
     journal={arXiv e-prints},
       pages={arXiv:2202.05706},
      eprint={2202.05706},
}

\bib{DAdderio-Iraci-VandenWyngaerd-TheBible-2019}{article}{
      author={D'Adderio, Michele},
      author={Iraci, Alessandro},
      author={Vanden~Wyngaerd, Anna},
       title={Decorated {D}yck paths, polyominoes, and the {D}elta conjecture},
        date={2018},
     journal={Mem. Amer. Math. Soc.},
}

\bib{DAdderio-Iraci-VandenWyngaerd-GenDeltaSchroeder-2019}{article}{
      author={D'Adderio, Michele},
      author={Iraci, Alessandro},
      author={Vanden~Wyngaerd, Anna},
       title={{The Schr{\"o}der case of the generalized Delta conjecture}},
        date={2019},
        ISSN={0195-6698},
     journal={Eur J Combin},
      volume={81},
       pages={58 \ndash  83},
  url={http://www.sciencedirect.com/science/article/pii/S0195669819300447},
}

\bib{DAdderio-Iraci-VandenWyngaerd-Theta-2021}{article}{
      author={D'Adderio, Michele},
      author={Iraci, Alessandro},
      author={Vanden~Wyngaerd, Anna},
       title={Theta operators, refined {Delta} conjectures, and coinvariants},
    language={en},
        date={2021-01},
        ISSN={0001-8708},
     journal={Adv Math},
      volume={376},
       pages={107447},
  url={https://www.sciencedirect.com/science/article/pii/S0001870820304758},
}

\bib{DAdderio-Mellit-Compositional-Delta-2020}{article}{
      author={{D'Adderio}, Michele},
      author={{Mellit}, Anton},
       title={A proof of the compositional {D}elta conjecture},
        date={2020},
     journal={ArXiv e-prints},
}

\bib{DAdderio-Romero-Theta-Identities-2020}{article}{
      author={{D'Adderio}, Michele},
      author={Romero, Marino},
       title={New identities for {T}heta operators},
        date={2020},
     journal={ArXiv e-prints},
}

\bib{Dhar-1990}{article}{
      author={Dhar, Deepak},
       title={Self-organized critical state of sandpile automaton models},
        date={1990-04},
     journal={Physical Review Letters},
      volume={64},
      number={14},
       pages={1613\ndash 1616},
         url={https://link.aps.org/doi/10.1103/PhysRevLett.64.1613},
}

\bib{Egecioglu-Remmel-Bricks}{article}{
      author={Egecioglu, \"{O}mer},
      author={Remmel, Jeffrey~B.},
       title={Brick tabloids and the connection matrices between bases of
  symmetric functions.},
        date={1991},
     journal={Discret. Appl. Math.},
      volume={34},
      number={1-3},
       pages={107\ndash 120},
         url={http://dblp.uni-trier.de/db/journals/dam/dam34.html},
}

\bib{Garsia-Haiman-qLagrange-1996}{article}{
      author={Garsia, Adriano~M.},
      author={Haiman, Mark},
       title={A remarkable {$q,t$}-{C}atalan sequence and {$q$}-{L}agrange
  inversion},
        date={1996},
        ISSN={0925-9899},
     journal={J Algebr Comb},
      volume={5},
      number={3},
       pages={191\ndash 244},
         url={https://doi.org/10.1023/A:1022476211638},
      review={\MR{1394305}},
}

\bib{Haglund-Book-2008}{book}{
      author={Haglund, James},
       title={The {$q$},{$t$}-{C}atalan numbers and the space of diagonal
  harmonics},
      series={University Lecture Series},
   publisher={American Mathematical Society, Providence, RI},
        date={2008},
      volume={41},
        ISBN={978-0-8218-4411-3; 0-8218-4411-3},
        note={With an appendix on the combinatorics of Macdonald polynomials},
      review={\MR{2371044}},
}

\bib{Haglund-Haiman-Loehr}{article}{
      author={Haglund, James},
      author={Haiman, Mark},
      author={Loehr, Nicholas},
       title={A combinatorial formula for {M}acdonald polynomials},
        date={2005},
        ISSN={0894-0347},
     journal={J. Amer. Math. Soc.},
      volume={18},
      number={3},
       pages={735\ndash 761},
         url={https://doi.org/10.1090/S0894-0347-05-00485-6},
      review={\MR{2138143}},
}

\bib{HHLRU-2005}{article}{
      author={Haglund, James},
      author={Haiman, Mark},
      author={Loehr, Nicholas},
      author={Remmel, Jeffrey~B.},
      author={Ulyanov, Anatoly},
       title={A combinatorial formula for the character of the diagonal
  coinvariants},
        date={2005},
        ISSN={0012-7094},
     journal={Duke Math J},
      volume={126},
      number={2},
       pages={195\ndash 232},
         url={https://doi.org/10.1215/S0012-7094-04-12621-1},
      review={\MR{2115257}},
}

\bib{Haglund-Morse-Zabrocki-2012}{article}{
      author={Haglund, James},
      author={Morse, Jennifer},
      author={Zabrocki, Mike},
       title={A {C}ompositional {S}huffle {C}onjecture {S}pecifying {T}ouch
  {P}oints of the {D}yck {P}ath},
        date={2012},
        ISSN={0008-414X},
     journal={Canadian J Math},
      volume={64},
      number={4},
       pages={822\ndash 844},
         url={https://doi.org/10.4153/CJM-2011-078-4},
      review={\MR{2957232}},
}

\bib{Haglund-Remmel-Wilson-2018}{article}{
      author={Haglund, James},
      author={Remmel, Jeffrey~B.},
      author={Wilson, Andrew~T.},
       title={The {D}elta {C}onjecture},
        date={2018},
        ISSN={0002-9947},
     journal={Trans. Amer. Math. Soc.},
      volume={370},
      number={6},
       pages={4029\ndash 4057},
         url={https://doi.org/10.1090/tran/7096},
      review={\MR{3811519}},
}

\bib{Haiman-Vanishing-2002}{article}{
      author={Haiman, Mark},
       title={Vanishing theorems and character formulas for the {H}ilbert
  scheme of points in the plane},
        date={2002},
        ISSN={0020-9910},
     journal={Invent Math},
      volume={149},
      number={2},
       pages={371\ndash 407},
         url={https://doi.org/10.1007/s002220200219},
      review={\MR{1918676}},
}

\bib{Hicks-Romero-2018}{article}{
      author={Hicks, Angela},
      author={Romero, Marino},
       title={{Delta operators at $q = 1$ and polyominoes}},
        date={2018},
     journal={{Sem. Lothar. Combin.}},
      volume={{80B}},
       pages={{12 pp.}},
}

\bib{Hogancamp-2017}{techreport}{
      author={Hogancamp, Matthew},
       title={Khovanov-{Rozansky} homology and higher {Catalan} sequences},
        date={2017},
         url={https://ui.adsabs.harvard.edu/abs/2017arXiv170401562H},
        note={ADS Bibcode: 2017arXiv170401562H Type: article},
}

\bib{Iraci-Rhoades-Romero-2022}{article}{
      author={Iraci, Alessandro},
      author={Rhoades, Brendon},
      author={Romero, Marino},
       title={A proof of the fermionic {Theta} coinvariant conjecture},
        date={2022-02},
     journal={arXiv:2202.04170 [math]},
         url={http://arxiv.org/abs/2202.04170},
        note={arXiv: 2202.04170},
}

\bib{Lascoux-Leclerc-Thibon-1997}{article}{
      author={Lascoux, Alain},
      author={Leclerc, Bernard},
      author={Thibon, Jean-Yves},
       title={Ribbon tableaux, {Hall}–{Littlewood} functions, quantum affine
  algebras, and unipotent varieties},
        date={1997-02},
        ISSN={0022-2488},
     journal={Journal of Mathematical Physics},
      volume={38},
      number={2},
       pages={1041\ndash 1068},
         url={https://aip.scitation.org/doi/10.1063/1.531807},
}

\bib{Macdonald-Book-1995}{book}{
      author={Macdonald, Ian~G.},
       title={Symmetric functions and {H}all polynomials},
     edition={Second},
      series={Oxford Mathematical Monographs},
   publisher={The Clarendon Press, Oxford University Press, New York},
        date={1995},
        ISBN={0-19-853489-2},
        note={With contributions by A. Zelevinsky, Oxford Science
  Publications},
      review={\MR{1354144}},
}

\bib{Mellit-Torus-Knots-2017}{article}{
      author={Mellit, Anton},
       title={Homology of torus knots},
        date={2017-04},
     journal={arXiv:1704.07630 [math]},
         url={http://arxiv.org/abs/1704.07630},
        note={arXiv: 1704.07630},
}

\bib{Mellit-Rational-2021}{article}{
      author={Mellit, Anton},
       title={Toric braids and (m,n)-parking functions},
        date={2021-12},
        ISSN={0012-7094, 1547-7398},
     journal={Duke Mathematical Journal},
      volume={170},
      number={18},
       pages={4123\ndash 4169},
  url={https://projecteuclid.org/journals/duke-mathematical-journal/volume-170/issue-18/Toric-braids-and-mn-parking-functions/10.1215/00127094-2021-0011.full},
}

\bib{Perkinson-Yang-Yu-2017}{article}{
      author={Perkinson, David},
      author={Yang, Qiaoyu},
      author={Yu, Kuai},
       title={G-parking functions and tree inversions},
    language={en},
        date={2017-04},
        ISSN={1439-6912},
     journal={Combinatorica},
      volume={37},
      number={2},
       pages={269\ndash 282},
         url={https://doi.org/10.1007/s00493-015-3191-y},
}

\bib{Rhoades-Wilson-2023}{techreport}{
      author={Rhoades, Brendon},
      author={Wilson, Andy},
       title={The {Hilbert} series of the superspace coinvariant ring},
        date={2023},
         url={http://arxiv.org/abs/2301.09763},
        note={arXiv:2301.09763 [math] type: article},
}

\bib{Romero-Thesis}{thesis}{
      author={Romero, Marino},
       title={{Delta eigenoperators and e-positivities in the theory of
  Macdonald polynomials}},
        type={Ph.D. Thesis},
        date={{2019}},
}

\bib{Stanley-Book-1999}{book}{
      author={Stanley, Richard~P.},
       title={Enumerative combinatorics. {V}ol. 2},
      series={Cambridge Studies in Advanced Mathematics},
   publisher={Cambridge University Press, Cambridge},
        date={1999},
      volume={62},
        ISBN={0-521-56069-1; 0-521-78987-7},
         url={https://doi.org/10.1017/CBO9780511609589},
        note={With a foreword by Gian-Carlo Rota and appendix 1 by Sergey
  Fomin},
      review={\MR{1676282}},
}

\end{biblist}
\end{bibdiv}

\end{document}